\documentclass{amsart}[12 pt]
\usepackage{amsmath, amsfonts, amssymb, euscript, amscd, latexsym, mathrsfs, bm, color, bbm, tabularx, xfrac, nicefrac, relsize}
\usepackage[left=1.6in, right=1.6in]{geometry}
 \usepackage{titlesec}
 \usepackage[pdftex,colorlinks=true,
                       pdfstartview=FitV,
                       linkcolor=blue,
                       citecolor=blue,
                       urlcolor=blue,
           ]{hyperref}
           
 \usepackage{enumitem}

\def\aa#1{ \begin{align*} #1 \end{align*} }
\def\aaa#1{ \begin{align} #1 \end{align} }
\def\mm#1{ \begin{multline*} #1 \end{multline*} }
\def\mmm#1{ \begin{multline} #1 \end{multline} }

\def\begeq{\begin{equation} \begin{cases}} 
\def\endeq{ \end{cases} \end{equation}}
\def\eq#1{ \begeq #1 \endeq }

\def\bege{\begin{equation*} \begin{cases}} 
\def\ende{ \end{cases} \end{equation*}}
\def\eqq#1{ \bege #1 \ende}
\def\eqn#1{ \begin{equation} \begin{aligned} #1 \end{aligned}\end{equation}}
\def\eqm#1{ \begin{equation*} \begin{aligned} #1 \end{aligned}\end{equation*}}

\newtheorem{thm}{\sc Theorem}[section]
\newtheorem{lem}{\sc Lemma}[section]
\newtheorem{cor}{\sc Corollary}[section]

\newtheorem{pro}{\sc Proposition}[section]
\newtheorem{rem}{\sc Remark}[section]

\newtheorem{innercustomthm}{\sc Theorem}
\newenvironment{customthm}[1]
  {\renewcommand\theinnercustomthm{#1}\innercustomthm}
  {\endinnercustomthm}

\newcommand{\eps}{\varepsilon}
\newcommand{\pl}{\partial}
\newcommand{\gt}{\geqslant}
\newcommand{\lt}{\leqslant}
\newcommand{\te}{\theta}
\newcommand{\sub}{\subset}
\newcommand{\dl}{\delta}
\newcommand{\al}{\alpha}
\newcommand{\gm}{\gamma}
 \newcommand{\Gm}{\Gamma}

 \newcommand{\la}{\lambda}
 \newcommand{\La}{\Lambda}
 \newcommand{\sg}{\sigma}
  \newcommand{\Sg}{\Sigma}

\newcommand{\om}{\omega}
\newcommand{\mc}{\mathcal}
 
\newcommand{\ms}{\mathscr}
\newcommand{\Om}{\Omega}
\newcommand{\mf}{\mathfrak}

\newcommand{\C}{{\rm C}}

\newcommand{\td}{\tilde}

\newcommand{\x}{\times}
\newcommand{\mto}{\mapsto}

\newcommand{\lm}{{\sss \lambda}}
\newcommand{\rf}{\eqref}
\newcommand{\bi}{\begin{itemize}}
\newcommand{\ei}{\end{itemize}}

\DeclareMathOperator{\ind}{\mathbbm{1}}
\newcommand{\lap}{\Delta}

\newcommand{\lb}{\label}

\newcommand{\fdot}{\,\cdot\,}

\newcommand{\sss}{\scriptscriptstyle}

\DeclareMathOperator{\ctg}{\text{ctg}}

\DeclareMathOperator{\arcctg}{\text{arcctg}}

\newcommand{\sig}{\varsigma}
\newcommand{\tet}{\vartheta}

\def\Rnu{{\mathbb R}}

\def\Nnu{{\mathbb N}}

\def\ffi{\varphi}

\def\com#1{}

\long\def\symbolfootnote[#1]#2{\begingroup%
\def\thefootnote{\fnsymbol{footnote}}\footnote[#1]{#2}\endgroup}
\titleformat{\section}[hang]{\large\bfseries}{\thesection.}{1ex}{}{}
\titleformat{\subsection}[hang]{\normalsize\bfseries}{\thesubsection}{2ex}{}{}
\titleformat{\subsubsection}[hang]{\small\bfseries}{\thesubsubsection}{2ex}{}{}

\allowdisplaybreaks

\usepackage{lipsum}
\usepackage{algorithmic}

\title[Singular solutions to $k$-Hessian equations]{Singular solutions to $\bm k$-Hessian equations \\ with fast-growing nonlinearities}

\begin{document}

\author{Jo\~ao Marcos do \'O$^{*}$}\thanks{$^{*}$Corresponding Author}

\address{Departamento de Matem\'atica, Universidade Federal da Para\'iba, Jo\~ao Pessoa, Brazil}
\email{jmbo@pq.cnpq.br}

\author{Evelina Shamarova}

\address{Departamento de Matem\'atica, Universidade Federal da Para\'iba, Jo\~ao Pessoa, Brazil}
\email{evelina.shamarova@academico.ufpb.br}

\author{Esteban da Silva}

\address{Departamento de Matem\'atica, Universidade Federal do Rio Grande do Norte, Natal, Brazil}
 \email{esteban.silva@ufrn.br}

\maketitle

\vspace{-8mm}

 \begin{abstract}
We study a class of elliptic problems, involving a $k$-Hessian and a  very fast-growing nonlinearity, on a unit ball. 
We prove the existence of a radial singular solution and obtain its exact asymptotic behavior
in a neighborhood of the origin. 
Furthermore, we study the multiplicity of regular solutions and bifurcation diagrams. 
 An essential ingredient of this study is analyzing the number of intersection points between the singular and regular solutions for rescaled problems. In the particular case of the exponential nonlinearity, we obtain the convergence of regular solutions to the singular and analyze the intersection number depending on the parameter $k$ and the dimension $d$.
 \end{abstract}

\vspace{4mm}

{\footnotesize
{\noindent \bf Keywords:} 
$k$-Hessian equation; singular solution; multiplicity of regular solutions; 
number of intersection points; bifurcation diagrams; 
Liouville-Bratu-Gelfand problem
\vspace{1mm}

{\noindent\bf 2020 MSC:}  34B15, 34C23, 35A24, 35J15, 35J62, 35J92

}

\section{Introduction}
In this paper, we are concerned with the following Dirichlet problem:
\eq{
\lb{G0}
S_k(D^2 w) = \la \, F(-w) &  \quad  \text{ in }\Om,\\ 
w<0, &  \quad  \text{ in }\Om,\\
w = 0,  &  \quad \text{ on }   \pl \Om,
}
where $\Om\sub\Rnu^d$ is a bounded domain, $F$ is a positive function increasing to infinity, whose exact properties
will be specified later, and $\la>0$ is a parameter.
Above, $1\lt k\lt d$ is an integer, where $d > 2$ is the dimension, and $S_k(D^2w)$ is the $k$-Hessian, 
i.e.,  the symmetric elementary function of the $k$-th order of the eigenvalues of
$D^2 w$. 

We prove the existence of a specific radial singular solution to \rf{G0} in the case 
when $\Om$ is a unit ball and obtain its exact asymptotic behavior in a neighborhood of the origin.   We furthermore demonstrate that
this specific singular solution allows to study properties of radial regular solutions to \rf{G0} such as bifurcation diagrams and
multiplicity, where an important tool is the number of intersection points 
 between the singular and regular solutions for rescaled problems.


\subsection{Elliptic problems with nonlinearities increasing to infinity}
In recent years, semilinear elliptic problems of the form 
 \eq{
 	\lb{SEP}
 	-\lap u = \la G(u), &  \quad  \text{ in }\Om,\\
 	u = 0,  &  \quad \text{ on }   \pl \Om,
 }
where $G$ is smooth and increasing to infinity,  have been of significant interest; see, e.g.,  
\cite{MR1605678, budd, cabre, MR0382848, Dolb, MR2779463,  gelfand, ghergu, jacob-schmitt, josef-lundgren, kikuchi, miyamoto2015, miyamoto2018, MR1779693}. This list of references is by no means complete, so we refer the reader to \cite{MR1605678, cabre, MR0382848,josef-lundgren,MR1779693}, and monograph \cite{MR2779463}
for a survey of results on this subject.
The following important points are usually addressed when dealing with this type problems:
1) the existence of a critical parameter $\lambda^{\#}$ 
such that  for each $\lambda  \in (0, \lambda^{\#})$ there exists a 
regular solution $(u_{\lm},\la)$ to problem \rf{SEP}, and
there are no solutions to \rf{SEP} for $ \lambda >  \lambda^{\#}$; 2) the existence of  a parameter $\lambda^*\in (0 , \lambda^{\#}]$ such that 
problem \eqref{SEP} possesses a singular solution for $\lambda=\lambda^*\!.$ 
Since a tremendous amount of papers (see, e.g., the above-cited ones and references therein) has been dedicated to problems of
type \rf{SEP} and, in particular, to answering the above questions, this topic can be regarded as fundamental.

In this work, we are concerned with a related problem; namely, with problem \rf{G0} involving a $k$-Hessian
and a function $F$ under quite flexible imposed assumptions. 
$k$-Hessian equations constitute an important class of fully nonlinear PDEs; so they have been studied 
by many authors \cite{chou,dai, oliveira-jmbo, sanchez1, sanchez2,tso,trudinger,wei1,wei2}. 
When restricted to so-called
$k$-admissible solutions (for the definition see e.g. \cite{tso}), these equations are elliptic and their solutions enjoy properties similar to those of solutions to semilinear equations involving the Laplacian.
 Furthermore, $k$-Hessian equations are of interest in geometric PDEs \cite{guan} and differential geometry \cite{shen-trudinger}.

Problem \rf{G0} is in turn a larger generalization, compared to those that exist in 
the literature \cite{ghergu, jacob-schmitt, kikuchi,miyamoto2015},
of the classical Gelfand problem. The latter deals
with positive solutions to the system
\eq{
\lb{gelfand0}
\lap u + \la\, e^u = 0 &  \quad    \text{ in }\Om,\\
u = 0  &  \quad  \text{ on }   \pl \Om.
}
Problem \rf{gelfand0} is a fundamental problem that
arises in many theories, such as non-linear diffusions generated by non-linear sources  \cite{1,3,2}, 
 thermal ignition of a mixture of chemically active gases \cite{gelfand},
unstable membrane \cite{5}, and gravitational equilibrium of polytropic stars   \cite{8,7}, among others.
For problem \rf{gelfand0}, one is mainly interested in the existence and multiplicity of regular solutions, 
as well as bifurcation diagrams.
 Most of authors
were concentrated on the case of a unit ball, e.g., \cite{gelfand, josef-lundgren},
since in this case, it is known that all solutions are radial \cite{gidas} and problem
\rf{gelfand0} has the following simple representation:
\eq{
\lb{gelfand}
 u''  + \frac{d-1}r u' + \la e^u = 0, \\
 u'(0) = u(1) = 0.
}
In \cite{gelfand}, Gelfand proved the existence of the value $\la$ for which \rf{gelfand} has an infinite number of solutions.
In \cite{josef-lundgren}, the authors completely described the relation between the multiplicity of solutions to \rf{gelfand} and the dimension $d$
of the ball.

Problem \rf{gelfand} was further generalized by various authors in the following two directions: replacing $e^u$ with a more
general function  while keeping the radial Laplacian \cite{ghergu, kikuchi, miyamoto2015, miyamoto2018}; and 
replacing the Laplacian with the radial operator
\aaa{
\lb{L-op}
L(u) = r^{-\gm} (r^\al |u'(r)|^\beta u'(r))'
}
while preserving the term $e^u$  \cite{jacob-schmitt}. More specifically, in \cite{ghergu}, the term $e^u$ was replaced with the   $n$-times  iterated
exponential   (henceforth denoted by $\exp^{\circ n}(u)$),  in \cite{kikuchi}, it was replaced with $e^{u^p}\!,$ and in \cite{miyamoto2015}, this term 
was replaced with $e^u+\text{``lower-order term''}$.

\subsection{Our setting: operator $L$ and function $F$}
\lb{p-def}
In this work, 
we are concerned with negative radial solutions to problem \rf{G0} on a unit ball. 
At the same time, our methods allow to study the convergence of positive regular solutions to the singular for the equation
$L(u)+\la e^u = 0$. To hit both of these goals, in \rf{G0}, we make the substitution $w=-u$, and by this, 
we reduce problem \rf{G0} to its radial form (see Appendix for details) as follows:
\eq{
\lb{G}
 \tag{$P_{\lambda}$}
L(u) + \la \, e^{f(u)} = 0, \quad 0 \lt r < 1, \\
u>0, \quad  0 \lt r < 1,\\
u(1) = 0,  
}
where $L$ is defined by \rf{L-op} and $f=\ln F$ satisfies the assumptions below. For the sake of clarity, we formulate 
these assumptions (in terms of $f$ and the inverse function $g=f^{-1}$) in a simplified manner. 
For the actual and more general set of assumptions, we redirect the reader to Subsection \ref{s2.2}.
\begin{itemize}
\item[(1)] $f$ is smooth and increasing to infinity; 
\item[(2)] $|g''(t)|$ monotonically decreases to zero;
\item[(3)] $\lim_{t\to+\infty}\frac{f(g(t) + \eps(t))}{t} = 1$, where $\eps(t) = O(g''(s))$;
\item[(4)] As $t\to +\infty$,
$\frac{g'(t)}{g'(t+o(t))}\to 1$ and $\frac{g''(t)}{g''(t+o(t))} = O(1)$;
\item[(5)] $\lim_{u\to +\infty} f'(u) = +\infty$, 
$\lim_{u\to+\infty}  \frac{f''(u)}{f'(u)^2}(\ln f'(u))^2 = 0$;
 \item[(6)]  $\al>\beta+1$ (i.e., $d>2k$ for the $k$-Hessian); $\beta\gt 0$; $\gm+1>\al$.
  \end{itemize}
We emphasize that problem \rf{G}
is an extension of  problem \rf{gelfand}
in the two aforementioned directions simultaneously, i.e., replacing the Laplacian with the operator \rf{L-op}, 
which can be roughly regarded as the radial case of a $k$-Hessian, and replacing $\exp(u)$ with a rather general term $\exp\{f(u)\}$,
which encompasses exponential-type nonlinearities previously considered in the literature \cite{ghergu,kikuchi,miyamoto2015}, 
such as $\exp(u)$, $\exp(u^p)$,
the iterated exponential $\exp^{\circ n}(u)$, $\exp(u)+\text{``lower-order term''}$
(see Subsection \ref{examples}).  In addition,
due to flexibility of our assumptions, $\exp\{f(u)\}$ provides a plenty of new examples of nonlinearities 
which have not been considered in the literature, but can be treated within our framework.

We allow the parameters $\al$, $\beta$, and $\gm$ to take a larger range of values than it is required  for 
a $k$-Hessian since it does not give us extra work in the proofs and allows to include the important
case of a $p$-Laplacian. The exact values of $\al$, $\beta$, and $\gm$ for a $k$-Hessian and a
$p$-Laplacian are summarized in the table below. The table includes the parameter 
$\te = \gm +2+ \beta - \al$ whose role will become clear later.
\begin{table}[htbp]
\centering
\begin{tabularx}{0.55\textwidth}{X  X  X  X  X X} 
\hline
 & & $\al$ & $\beta$ & $\gm$ & $\te$ \\
\hline
\text{$k$-Hessian} & & $d-k$ & $k-1$ & $d-1$& $2k$ \\
\hline
\text{$p$-Laplacian} &  & $d-1$ & $p-2$ & $d-1$ & $p$ \\
\hline
\end{tabularx}
\label{table:tab1} 
\end{table}
\begin{rem}
 \rm
 Remark that since the solutions of interest to problem \rf{G0} are negative, according to \cite{tso},
 the operator $S_k(D^2 w)$ (regarded as a $d\x d$ matrix) is elliptic on those solutions.
 On the other hand, if $u$ is a solution to problem \rf{G}, 
 the operator $S_k(D^2 u(|\cdot|))$ may not be elliptic since $S_k(D^2 u) = (-1)^k S_k(D^2 w)$.
 Its ellipticity depends on the parity of $k$. 
 \end{rem}
\subsection{Contributions of this work and comparison with other results}
One of the main results of this work is a construction of a 
negative radial singular solution to problem \rf{G0}  on a unit ball, which is reduced to a construction of a 
positive radial singular solution to problem \rf{G} (see Appendix),
with an exact asymptotic behavior in a neighborhood of the origin. 
By a  positive singular solution to problem \rf{G0} in a ball, we understand a 
$\C^2$-function which solves problem \rf{G0} at all points of the ball except for the origin,
and tends to the negative infinity as the argument goes to $0$.
By a singular solution to problem \rf{G}, we understand a $\C^2$-function 
which solves problem \rf{G} at all points of $(0,1]$ and tends to the positive infinity
as the argument goes to $0+$.

Radial singular solutions on a unit ball have attracted significant interest from many researchers
 \cite{budd, ghergu, ghergu-mi, kikuchi, lin1994, merle, miyamoto2015, miyamoto2017, miyamoto2018, miyamoto2021}.  Most of these works deal however with semilinear equations involving the Laplacian.
 Papers \cite{ghergu-mi} and \cite{miyamoto2021} study radial singular solutions to equations
 with the operator \rf{L-op}, but they do not discuss asymptotics for these solutions,
 and furthermore, they study different types of nonlinearities.
 Among applications of singular solutions are bifurcation
 diagrams for elliptic problems \cite{ghergu, kikuchi, miyamoto2015, miyamoto2017, miyamoto2018} and 
 blow-up behavior of solutions to parabolic problems \cite{matano2004, tello2006}. 
Asymptotic representations for radial singular solutions on a unit ball were previously
obtained in  \cite{ghergu, kikuchi, merle, miyamoto2015}.
However,  these papers 
deal  with the radial Laplacian and very particular types of the function $f(u)$.
Moreover, the asymptotics of
singular solutions  obtained in  \cite{ghergu, kikuchi, miyamoto2015} are included in ours as particular cases.
Our assumptions on $f$ are quite flexible, so examples of $f$ include not only  $\exp^{\circ n}(u)$,
$u^p$, and $u+\text{``lower-order term''}$ (cf. \cite{ghergu, kikuchi, miyamoto2015}), but also many others. 
As such, we show that $f(u) = \exp^{\circ n}(u^p)$, satisfies our assumptions, 
encompassing both $\exp^{\circ n}(u)$ from \cite{ghergu} and $u^p$ from \cite{kikuchi}.
We also remark that,
compared to \cite{lin1994, miyamoto2017, miyamoto2018}, our 
results are valid for the operator $L$, given by \rf{L-op} and therefore generalizing the Laplacian,   
and, on the whole, our work contains a sharp asymptotic representation
of a singular solution  which is not present in  the aforementioned works.

We state our main theorem on the asymptotic representation of a radial singular solution as follows:
\begin{thm}
\lb{thm1-}
Assume (1)--(6). Then,
there exists $\la^*>0$ such that problem \rf{G} with $\la=\la^*$
possesses a radial singular solution of the form
\eqn{
\lb{as1}
&u^*_{\lm^{\!*}}(r)  =  g(Z(r)) + O\big(g''(\ln r^{-\te})\big),  \quad \text{where}\\
&Z(r) =  
 \ln \big\{ \te^{\beta+1}(\al-\beta-1)\} - \ln \la^*-\te \ln r \\
&\hspace{7mm} + (\beta+1)\ln\{ g'(\tau)+(\beta+1)g''(\tau) \ln g'(\tau) \big\}
 \quad\text{with} \; \; \tau = \ln\Big\{ \frac{\beta+1}{\la^*r^\te}\Big\}.
 }
\end{thm}
The important particular case is the exact form for the singular solution to problem \rf{G} in the case $f(u) = u$. 
Namely, the following corollary holds.
\begin{cor}
\lb{cor1111}
  Assume (6) and consider problem  \rf{G} with $f(u) = u$.
  Then, $(u^*_{\lm^{\!*}},\la^*)$, given by 
  \aa{
  u^*_{\lm^{\!*}}(x) = -\te \ln|x|, \quad \la^* =  \te^{\beta+1}(\al - \beta -1),
  }
is a radial singular solution to problem \rf{G}.
\end{cor}
We furthermore use the singular solution $(u^*_{\lm^{\!*}},\la^*)$, constructed  in Theorem \ref{thm1-},
 to study properties of regular solutions such as 
 bifurcation diagrams and multiplicity. Specifically, we obtain the following result.
\begin{thm}
\lb{thm1111}
Assume (1)--(6) and let
$\al-\beta-1< \frac{4\te}{\beta+1}$ $(d<2k+8)$. Further, we let
 $(u_{\sss\lm(\rho)}, \la(\rho))$  be  the family of regular solutions to problem \rf{G} parametrized by $\rho =\sup_{r\lt 1} |u_\lm(r)|$, and let  
 $(u^*_{\lm^{\!*}}, \la^*)$ be the singular solution constructed in Theorem \ref{thm1-}.
 Then, as $\rho\to +\infty$, $\la(\rho)$ oscillates around $\la^*$. In particular, there are 
 infinitely many regular solutions to problem \rf{G} for $\la(\rho) = \la^*$. 
\end{thm}
An important tool for obtaining the above result
is the number of intersection points between the regular and singular solutions to the equation
\aaa{
 \lb{sing1}
L(u) + e^{f(u)} = 0.
}
Note that $u^*(r) = u^*_{\lm^{\!*}}(r (\la^*)^{-\frac1{\te}})$ is a singular solution to \rf{sing1}. In general, the change of variable
\aaa{
\lb{ch-var}
u(r) = u_\lm(r \la^{-\frac1{\te}})
} 
brings the equation in \rf{G}, containing $\la$, to equation \rf{sing1}. As such,
 the radial regular solution $u(\fdot,\rho)$ to \rf{sing1} 
 is obtained from the radial regular solution $(u_{\sss\lm(\rho)}, \la(\rho))$ 
 to \rf{G} 
by the change of variable \rf{ch-var}. We have the following result.
\begin{thm}
\lb{thm2222}
Assume (1)--(6) and let
 $\al-\beta-1< \frac{4\te}{\beta+1}$.  
  Then, for any $\dl>0$, it holds that
 $\lim_{\rho\to +\infty} \ms L_{(0,\dl)}(u(\fdot,\rho)- u^*(\fdot)) = +\infty$, where $u$ and $u^*$ are regular and singular solutions
 to \rf{sing1} and
 $\ms L_{(0,\dl)}(u-u^*)$ is the number of zeros of $u-u^*$ on the interval $(0,\dl)$.
\end{thm}
Furthermore, we turn our attention to a study of radial regular solutions. The result that we obtain in this direction is 
the following. 
\begin{thm}
 \lb{thm3131}
 Assume (1), (5), and (6). Then, 
there exists a parameter $\la^\#>0$, depending on $\al$, $\beta$, $\gm$, and $f$, such that
any  regular solution $(u_\lm,\la)$ to problem \rf{G} satisfies $\la\lt \la^\#$. Moreover, for any $\la<\la^\#$,
there exists a regular solution to \rf{G}, and for $\la=\la^\#$, the solution can be either regular or singular.
\end{thm}
The actual set of assumptions for Theorem \ref{thm3131} is, again, somewhat weaker, and the reader is redirected
to Subsection \ref{s3.3}.

Finally, we study the equation
\aaa{
\lb{sing1111}
L(u) + e^u = 0
}
In \cite{jacob-schmitt}, the authors considered
problem  \rf{G} with the nonlinearity $e^u$ and
obtained  the result of Theorem \ref{thm1111} along with the convergence 
$\la(\rho)\to\la^*$ as $\rho\to+\infty$. 
In the current work, we complement the aforementioned
results  by studying the number of intersection points
between the regular and singular solutions to \rf{sing1111}, depending on $\al$, $\beta$, and $\gm$,
and by proving the convergence of regular solutions to the singular as $\rho\to +\infty$.
Namely, we have the following results.
\begin{thm}
\lb{nozeros} Assume (6).
Suppose $0<\al-\beta-1< \frac{4\te}{\beta+1}$ $(2k<d<2k+8)$. Then,  $\ms L_{(0,+\infty)}(u-u^*) = +\infty$.
 \end{thm}
\begin{thm}
 \lb{zeros2-}
 Assume (6).
 Suppose $\al-\beta-1\gt 4\te(\beta+1)$ $(d\gt 2k+8k^2)$. Then, $\ms L_{(0,+\infty)}(u-u^*) = 0$.
   \end{thm}
   \begin{thm}
   \lb{conv1122}
Assume (6). Then,
\aaa{
\lb{lim-1}
\lim_{\rho\to +\infty} u(\fdot,\rho) = u^* \quad \text{and} \quad 
\lim_{\rho\to +\infty} (u_{\sss\lm(\rho)} , \la(\rho)) = (u^*_{\lm^{\!*}}, \la^*) \quad 
 \text{in} \;\; \C_{loc}(0,+\infty).
}
\end{thm}
\begin{rem}
\rm
Remark that the results on the intersection number, similar to those 
of Theorems \ref{nozeros} and \ref{zeros2-}, are known
for the equation $L(u)+ |u|^{p-1} u = 0$  (see e.g. \cite{miyamoto2017}).
Here they are obtained for the exponential counterpart of this equation.
\end{rem}
\subsection{Methods}
In this work, we use combinations of newly developped methods along with extensions, from the Laplacian to 
the $k$-Hessian, of existing techniques.
 New methods are used for construction 
of a singular solution and obtaining its sharp asymptotic (Theorem \ref{thm1-}), for
proving the convergence of regular solutions to the singular (Theorem \ref{conv1122}), and
for showing that the regular and singular solutions to equation \rf{sing1111} do not intersect each other if
$\al-\beta-1\gt 4\te(\beta+1)$ (Theorem \ref{nozeros}). To be specific, in the proof of Theorem
\ref{thm1-}, we arrive at the equation
\aaa{
\lb{eq-g1}
v'' - (\hat \al -1) v' = - \frac{e^{-\te t + f(v)}}{|v'|^\beta}
}
which is equivalent to the equation in \rf{G}, but the variable $t$ varies 
in a neighbourhood of infinity.  The difficulty arises due to the fact that the denominator
on the right-hand side contains $|v'|^\beta$, the term which is not present in the Laplacian case 
($\beta = 0$). This does not allow to apply methods known for semilinear equations with the Laplacian
 (see, e.g., \cite{ghergu, kikuchi, miyamoto2015}). 
In particular, the singular solution $v^*$ to \rf{eq-g1} takes the form
\aaa{
\lb{sol-g1}
v^*(t) =  g(\te t +\ffi_1(t)+\ffi_2(t)) + \eta(t),
}
where 
\eqn{
\lb{ffi}
&\ffi_1 =  \ln \{\te^\beta\big(g'(\te t) + (\beta+1)g''(\te t) \ln g'(\te t)\big)^\beta\}, \\
&\ffi_2 = \ln\Big\{\te\Big(\frac\al{\beta+1} -1\Big)\Big\} + \ln \{g'(\te t) + (\beta+1)g''(\te t) \ln g'(\te t)\}
}
and $\eta(t)=  O(g''(\te t))$.
Remark that the term $\ffi_1$ does not appear in the Laplacian case;
it was for the first time introduced in this work.
Its role is to make certain cancelations with the denominator  $|{v^*}'|^\beta$ to ensure 
the possibility to reduce \rf{eq-g1} to a pair of semilinear equations with respect to $\eta$ and $\eta'$. The order of the reminder $\eta$, in the general setting, was
also first established in this work. Additional technical difficulties arise due
to the fact that the function $f$ (unlike  \cite{ghergu, kikuchi, miyamoto2015})
is arbitrary, so the asymptotic representation \rf{as1} is expressed via the derivatives of $f^{-1}$.

Furthermore, in the proof of Theorem  \ref{conv1122}, we use the result of Section \ref{s3} on the relative compactness of regular solutions on each 
compact interval of $(0,+\infty)$ (along with the result of \cite{jacob-schmitt}).
To the best of our knowledge, the aforementioned technique is new. Also, it is worth to mention Theorem \ref{zeros2-}
whose proof does not appear to have analogs in the literature; while the result itself, together
with Theorem \ref{nozeros}, can be regarded as fundamental. Indeed, Theorems \ref{nozeros}
and \ref{zeros2-} study the number of intersection points between the regular and singular
solutions for the canonical equation \rf{sing1111}, the result which did not appear in
the literature prior to this work.

Furthermore, many of our methods were previously unknown in the context of $k$-Hessian equations. They represent non-trivial extensions, from the Laplacian to the $k$-Hessian, of existing techniques.
For example, in the proof of Theorem \ref{thm2222}, dealing with the intersection number between the regular
and singular solutions, one obtains a certain equation via the following transformation of equation \rf{sing1}:
 \aaa{
  \lb{transf}
 \td u(s, \rho) = \mc F^{-1}_1 \eps_\rho^{-\frac\te{\beta+1}} \mc F(u(\eps_\rho s,\rho)), \quad \eps_\rho =  \Big(\frac{\mc F(\rho)}{\mc F_1(1)}\Big)^{(\beta+1)/\te},
 }
where  
 $\mc F(u) = \int_u^{+\infty} \exp\big\{\!\!-\!\frac{f(s)}{1+\beta}\big\} ds$  and 
 $\mc F_1(u) = (\beta+1) e^{-\frac{u}{\beta+1}}$.
We believe that in the $k$-Hessian case, this transformation, in its exact form \rf{transf}, 
was first used in this work. 
A version of this transformation, slightly different from \rf{transf}, was introduced in \cite{ghergu-mi}.
The ``Laplacian'' analog of \rf{transf} can be found in, e.g., \cite{miyamoto2018}. 
The important difference with the Laplacian case is the following: when we apply this tranformation, 
 the first two terms $L(\td u) + e^{\td u}$ of the transformed equation are exactly as in 
the canonical equation \rf{sing1111}; however, the additional 
term $s^{\al-\gm} |\td u'|^{\beta+2}\big(\frac{I(u(\eps_\rho s))}{\beta+1} - 1\big)$,
where $I(u) = \mc F(u) f'(u) \exp\{\frac{f(u)}{\beta+1}\}$,
has a more complicated structure and its
convergence to zero is rather difficult to show.
This convergence however is crucial to show the convergence
of $\td u$  to the regular solution of the canonical (limit) equation \rf{sing1111}. 
Remark that even in the Laplacian case, considered in \cite{miyamoto2018},
the aforementioned  convergence was announced without a proof.
It is important to mention that in order to show the convergence
of singular solutions, which is also crucial for proving Theorem \ref{thm2222}, 
we employ our novel asymptotic  \rf{as1} obtained in Theorem \ref{thm1-}.
Thus, we reinforce that the proof of Theorem \ref{thm2222} on
the intersection number differs significanly from the Laplacian case
and requires new tools developed in this work. 

Furthermore, the proof of Theorem \ref{thm1111} on bifurcation diagrams relies havily on 
Theorem \ref{thm2222} discussed above. In addition,
the aforementioned proof contains a much more detailed analysis 
on the oscillation of $\la(\rho)$ arround $\la^*$ compared to \cite{miyamoto2018} and preceeding works. 

Finally, we
 mention     Proposition \ref{lrhoB}   that deals with upper and lower bounds
 for $u(\fdot,\rho)^{-1}(B)$, $B>0$, whenever $\rho$ is sufficiently big. 
 In the Laplacian case, this result was obtained in \cite{lin1994} (Theorem 2.4) and then was widely 
 used in the same paper. 
 However, an extension  of Theorem 2.4 to the operator $L$
is as important for $k$-Hessian equations as the aforementioned theorem is important for
semilinear equations with the Laplacian.  
Besides of being a non-trivial extension from the Laplacian setting, Proposition
\ref{lrhoB}   is an important tool in the absence of which many of the results of this work could not be obtained.
As such,    Proposition \ref{lrhoB}  is used in the proof of the relative compactness of
regular solutions to equation \rf{sing1} (Proposition \ref{r-comp}) which is used to obtain
the convergence \rf{lim-1} in Theorem \ref{conv1122}. Next,    Proposition \ref{lrhoB}  is used to
obtain Theorem \ref{thm3131}.
Finally,    Proposition \ref{lrhoB}  is used in the following chain of results leading 
to one of the main results of this work. Namely, it is directly applied in Lemma \ref{lem334} implying Proposition 
\ref{lem98} on the convergence of the singular and regular solutions of
the equation obtained via  the transformation \rf{transf} to the respective solutions of the canonical equation \rf{sing1111}. 
This, in turn, implies Theorem \ref{thm2222} on the number of intersection points while the latter is used
in  the proof of Theorem \ref{thm1111} on bifurcation diagrams.

\subsection{Structure of this work}
 
Our paper is structured as follows. In Section \ref{s2}, we 
prove the existence of a singular solution to \rf{G} and
obtain its exact asymptotic behavior in a neighborhood of the origin. 
The main result of this section is Theorem \ref{thm1-}.
In Section \ref{s3}, we prove Theorem \ref{thm3131}. In the same section,
we obtain the existence and relative compactness of regular solutions to problem \rf{G}.
As a byproduct, we obtain
a version of Pohozaev's identity for the operator \rf{L-op}.
In Section \ref{s4}, we study bifurcation diagrams and the number of intersection
points between the regular and singular solutions for the case $0<\al-\beta-1\lt \frac{4\te}{\beta+1}$.
The main results of this section are Theorems \ref{thm1111} and \ref{thm2222}. It is important to mention
that we employ here the tranformation \rf{transf}
that reduces equation \rf{sing1} to an equivalent equation whose limit is  equation \rf{sing1111}. 
The aforementioned equation is  then studied in Section \ref{s5555}, where  
we obtain Theorems \ref{nozeros}, \ref{zeros2-},  \ref{conv1122}.

\section{Singular solution to Problem \rf{G}}
\lb{s2}

In this section, we will construct a positive radial singular solution to  \rf{G}.

\subsection{Standing assumptions and useful lemmas}
\lb{s2.2}
We start by introducing the 
 general set of assumptions implying assumptions (1)--(6) from Subsection \ref{p-def}.
Recall that $g$ denotes the inverse function for $f$.
\begin{itemize}
\item[(A1)] $f\in \C^4([0,+\infty))$, $f'>0$, and $\lim_{u\to+\infty} f(u) = +\infty$.
\item[(A2)] $\lim\limits_{t\to+\infty} g''(t) =0$ and $\lim\limits_{t\to+\infty}\frac{f(g(t) +\eps(t))}{t} = 1$, where  $\eps(t) = O(g''(t))$.
\item[(A3)]  
$\lim_{t\to+\infty}\frac{g'(t)}{g'(t+o(t))}= 1$.
\item[(A4)] $\lim_{t\to+\infty}  \frac{g''(t)}{g'(t)}\ln g'(t) = 0$.
\item[(A5)] 
As $t\to+\infty$,
\aa{
 &{\rm (a)} \;\;  g'''(t)(\ln g'(t))^n = O(g''(t)), \;\; n=0,1,2, \quad g^{(\rm{iv})}(t)\ln g'(t) = O(g''(t)),\\ 
 &{\rm (b)} \;\;  \sup\nolimits_{s\gt t}|g''(s)| =  O(g''(t)), \quad g''(t+o(t)) = O(g''(t)).
   }
  \item[(A6)]
  $\al>\beta+1$ (i.e., $d>2k$ for the $k$-Hessian); $\beta\gt 0$; $\te>0$.
  \end{itemize}

  \begin{rem}
  \lb{rem161}
  \rm
  It will be rather convenient to have some alternative equivalent conditions at hand
  for (A2) and (A4). Namely, (A2) and (A4) are equivalent to
\bi
\item[(A2')]  $\lim_{u\to+\infty} \frac{f''(u)}{f'(u)^3} =0$ and $\lim_{t\to+\infty}\frac{f(u+ \td\eps(u))}{f(u)} = 1$, where  
$\td \eps(u) = O(\frac{f''(u)}{f'(u)^3})$.
\item[(A4')] $\lim_{u\to+\infty}  \frac{f''(u)}{f'(u)^2}\ln f'(u)  = 0$;
  \ei
 We start showing (A4'). A straightforward computation gives
  \aa{
 \frac{f''(u)}{f'(u)^2}\ln f'(u) = -\frac{g''(t)}{g'(t)}\ln g'(t), \quad \text{where} \;\; u = g(t).
  }
  To show the equivalence of (A2) and (A2'), we note that $g''(f(u)) =\frac{f''(u)}{f'(u)^3}$, and hence, the first limits in (A2) and (A2') are equivalent.
Substituting $t=f(u)$,  we represent $\eta(t)$ as 
  $\phi(u) g''(f(u))$, where $\phi(u)$ is a bounded function. Defining $\td\eta(u) = \phi(u) \frac{f''(u)}{f'(u)^3}$, we obtain
  the desired equivalence.
  \end{rem}
   \begin{rem}
  \lb{rem141}
  \rm
 For the $p$-Laplacian, the condition $\al>\beta+1$ is interpreted as $p<d$.
  \end{rem}
 \begin{rem}
  \lb{rem16}
  \rm
  Note that under (A4) and (A2), $\lim_{t\to+\infty}  \frac{g''(t)}{g'(t)} = 0$. Indeed, assume the opposite. Then, there exists
  a sequence $t_n\to +\infty$ such that $\big| \frac{g''(t_n)}{g'(t_n)} \big|>M$ for some constant $M>0$. 
By (A4), it should hold that $\lim_{n\to+\infty} g'(t_n) = 1$, which, in turn, implies that  $\lim_{n\to+\infty}  \frac{g''(t_n)}{g'(t_n)} = 0$
by (A2).
\end{rem}
\begin{rem}
  \lb{rem26}
  \rm
 Some of examples of the function $\phi(t) = g''(t)$ satisfying  the first expression in (A5)-(b) are monotone functions and functions
 which can be represented as $\phi(t) = \phi_1(t) \phi_2(t)$, where $\phi_1$ is decreasing to zero
 as $t\to +\infty$,
 and $\phi_2$ is bounded from above and below, with an infinite number of maximum and minimum points, such as 
 $2+\sin t$ or $e^{\sin t}$. 
   \end{rem}
   \begin{rem}
   \rm
   \lb{A6}
 In what follows, we will need the following facts: as $t\to +\infty$,
   \aa{
   e^{\frac{t}{\beta+1}} g'(t) = O(t \,e^{\frac{t}{\beta+1}}), \qquad  e^{-\frac{t}{\beta+1}} g'(t) = O(t\, e^{-\frac{t}{\beta+1}}).
   }
   The above expressions obviously hold by (A2). Indeed, $g''(t)$ is bounded, and therefore, by Taylor's formula,
   $g'(t)$ has at most linear growth.
   \end{rem}
  
  It turns out that some of the conditions (A5)-(a) are fulfilled if $\lim_{u\to+\infty} f'(u) = +\infty$, or, which is the same, $\lim_{t\to+\infty} g'(t) = 0$.
\begin{lem}
\lb{lem9978}
Suppose $\lim_{u\to+\infty} f'(u) = +\infty$.
Then, conditions (A5)-(a), except for $n=2$, are fulfilled.
If, in addition,  $\lim_{u\to+\infty} \frac{f''(u)}{f'(u)^2}(\ln f'(u))^2 = 0$, then all conditions in 
 (A5)-(a) are fulfilled.
\end{lem}
\begin{proof}
Note that $g''(t)$ cannot be identically equal to zero in a neighborhood of $+\infty$. Indeed, in this case 
$g'(t) = 0$ in this neighborhood which is a contradiction.

Note that, by (A4), $\lim_{t\to\infty}  g''(t) \ln g'(t) = 0$. By L'Hopital's rule,
\aa{
\lim_{t\to+\infty}  \frac{g'''(t)}{g''(t)}\ln g'(t) 
= \lim_{t\to+\infty}  \frac{g''(t)}{g'(t)}\ln g'(t) -  \lim_{t\to+\infty}  \frac{g''(t)}{g'(t)} = 0.
}
The above inequality implies, in particular, that $\lim_{t\to+\infty}  g'''(t) \ln g'(t) = 0$.
Again, by L'Hopital's rule,
\aa{
\lim_{t\to+\infty}  \frac{g^{(4)}(t)}{g'''(t)}\ln g'(t) = \lim_{t\to+\infty}  \frac{g'''(t)}{g''(t)}\ln g'(t)  -   \lim_{t\to+\infty}  \frac{g''(t)}{g'(t)} = 0.
}
Note that $\lim_{t\to+\infty}  \frac{g'''(t)}{g''(t)} = 0$. Therefore,
\aa{
\lim_{t\to+\infty}  \frac{g^{(4)}(t)}{g''(t)}\ln g'(t) = \lim_{t\to+\infty}  \frac{g^{(4)}(t)}{g'''(t)}\ln g'(t) \lim_{t\to+\infty}  \frac{g'''(t)}{g''(t)} = 0.
}
Finally, a straightforward computation shows that
  \aa{
 \frac{f''(u)}{f'(u)^2}(\ln f'(u))^2 = \frac{g''(t)}{g'(t)}(\ln g'(t))^2, \quad \text{where} \;\; u = g(t).
  }
The first identity in (A5)-(a) for $n=2$ follows now by L'Hopital's rule.
\end{proof}
\begin{rem}
\rm
Remark that Lemma \ref{lem9978} justifies condition (5) in Subsection \ref{p-def}.
\end{rem}
\begin{lem}
\lb{lem223}
If $\zeta(t)$ is a differentiable function $\Rnu_+\to \Rnu$ 
such that $\lim_{t\to+\infty} \zeta'(t) = 0$, then $\lim_{t\to+\infty}\frac{\zeta(t)}{t} = 0$.
\end{lem}
\begin{proof}
We have
\aa{
\zeta(0) - \int_0^t|\zeta'(s)| ds \lt \zeta(t) \lt \zeta(0) + \int_0^t|\zeta'(s)| ds.
}
The integral $\int_0^t|\zeta'(s)| ds$ is either bounded or tends to $+\infty$ as $t\to +\infty$. In the first case the conclusion
of the lemma is obvious. In the second case it follows from L'Hopital's rule. 
\end{proof}
 \subsection{Examples of nonlinearities}
 \lb{examples}
Here, we give examples of functions $f(u)$ in \rf{G} satisfying (A1)--(A5).
 
\textit{Example 1. $f(u) = \exp^{\circ n}(u^p)$} ($p>0, n\gt 1$).
First, introduce the notation for the iterated exponential and the iterated logarithm:
 \aa{
 \exp^{\circ n}(t) = \underbrace{\exp\circ \exp\circ \dots \circ \exp}_n(t);\qquad
 \ln^{\circ n}(t) = \underbrace{\ln\circ\ln\circ\dots\circ\ln}_n(t).
 }
The inverse function  to $f$ is $g(t) =[\ln^{\circ n}(t)]^\frac1{p}$.
 We prove that (A1)--(A5) are fulfilled for the pair $f,g$. 
 \begin{rem}
 \rm
  Remark that in \cite{ghergu}, $f(u) = \exp^{\circ n}(u)$, and in \cite{kikuchi},
  $f(u) = u^p$, both are particular cases of $f(u) = \exp^{\circ n}(u^p)$ discussed in this example.
  In addition, we remark that \cite{ghergu} and \cite{kikuchi}  deal with the classical radial Laplacian, while
  our partial differential operator takes the form \rf{L-op}.
  \end{rem}
   First, we compute
  \aa{
 g'(t) &= \frac1{p} \,\frac{1}{ \ln^{\circ n}(t)^{1-\frac1{p}}\ln^{\circ (n-1)} (t)\ln^{\circ (n-2)}(t)\cdots t}\,,\\
  g''(t) &=  \frac1{p}\Big[ \Big(\frac1{p}-1\Big)\frac{1}{\ln^{\circ n} (t)^{2-\frac1{p}} \ln^{\circ (n-1)} (t)^2 \cdots t^2}
  -  \frac{1}{\ln^{\circ n} (t)^{1-\frac1{p}} \ln^{\circ (n-1)}(t)^2\cdots t^2}\\
  & - \frac1{\ln^{\circ n} (t)^{1-\frac1{p}} \ln^{\circ (n-1)} (t)\ln^{\circ (n-2)}(t)^2\cdots t^2}
  - \cdots - \frac{1}{\ln^{\circ n} (t)^{1-\frac1{p}} \ln^{\circ (n-1)} (t)\cdots t^2}\Big].
  }
 Note that $\lim_{t\to+\infty} g'(t) = 0$. Moreover,  $\lim_{t\to+\infty} \frac{g''(t)}{g'(t)}(\ln g'(t))^2 = 0$.
  To see the latter, we note that  $\frac{g''(t)}{g'(t)} = O(\frac1{t})$ and $\ln g'(t) = O(\ln(t))$. 
  Therefore, by Lemma \ref{lem9978}, (A5)-(a) is fulfilled. 
Also, $|g''(t)|$ is decreasing, so the first identity in (A5)-(b) is fulfilled. 
  The above argument also implies (A4).
   
   A verification of (A1) is straightforward. To verify (A3), it is sufficient to show that
    for the iterated logarithm,  it holds that $\lim_{t\to+\infty}\frac{\ln^{\circ n}(t)}{\ln^{\circ n}(t+o(t))} = 1$. 
   We prove the following statement by induction on $n$: $\ln^{\circ n}(t+o(t)) = \tau + o(\tau)$, where $\tau = \ln^{\circ n}(t)$.
   For $n=1$, it is obvious since $\ln(t+o(t)) = \ln t + o(1)$.
   Suppose we proved the statement for $n-1$, i.e., $\ln^{\circ (n-1)}(t+o(t)) = \tau + o(\tau)$ with $\tau = \ln^{\circ (n-1)}(t)$.
   Then, $\ln^{\circ n}(t) = \ln(\tau + o(\tau)) = \ln\tau + o(\ln(\tau))= \ln^{\circ n}(t) +o(\ln^{\circ n}(t))$,
which implies (A3). The same argument implies the second expression in (A5)-(b).
   
   It remains to verify (A2). 
  Here, without loss of generality we can set $p=1$. Indeed, define the function $\td g(t) =\ln^{\circ n}(t)$ so that $g(t)^p = \td g(t)$.
   We have 
   \aa{
   \big(g(t) + O(g''(t))\big)^p = \td g(t)\Big(1+\frac{ O(g''(t))}{g(t)}\Big)^p  =
   \td g(t) + O(g''(t))g(t)^{p-1}    
   = \td g(t) +  O(\td g''(t))
   }
   since $\td g''(t) = p\, g''(t) g(t)^{p-1}\big(1+(p-1)\frac{g'(t)^2}{g(t)g''(t)}\big) = g''(t) g(t)^{p-1} O(1)$. Indeed,
 $\frac{g''(t)}{g'(t)} = \frac1t + o(\frac1t)$ and $\frac{g'(t)}{g(t)} = O(\frac1t)$.
  We further argue by induction on $n$. For $n=1$, (A2) is clearly satisfied. Suppose (A2) is true for $n-1$. 
  We have
  \mm{
  \exp(\td g(t)+  O(\td g''(t))) = \ln^{\circ (n-1)}(t)\exp\{O(\td g''(t))\} =  \ln^{\circ (n-1)}(t)(1+ O(\td g''(t))) \\
  =  \bar g(t) + O(\bar g''(t)), \quad \text{where} \; \bar g(t) = \ln^{\circ (n-1)}(t).
  }
  Indeed, one immediately verifies that $\td g''(t)  \ln^{\circ (n-1)}(t) = \bar g''(t) + o(\bar g''(t))$.
(A2) now follows now from the induction hypothesis. 

   \textit{Example 2. $f(u) = u^p$}, $p>\frac12$.
We restrict here the range of $p$ to $(\frac12,+\infty)$ in order to guarantee that
$\lim_{t\to+\infty} g''(t) = 0$. Indeed,  $g''(t) = \frac1{p}(\frac1{p}-1) \, t^{\frac1{p}-2}$.
The verification of (A1)--(A5) is straightforward.

  \textit{Example 3. Perturbed Gelfand's problem.}  
  Take $f(u) = u + \tet(u)$, where $\tet(u)$ is a $\C^4$-smooth function 
  with the following properties:
  \bi
  \item[(i)]  $\lim_{u\to+\infty} \tet^{(n)}(u) = 0$ for $n=0,1,2$;
  \item[(ii)] $u+\tet(u)$ is increasing on $[0,+\infty)$;
  \item[(iii)] 
 $\tet'''(u) = O(\tet''(u))$,  $\tet^{(\rm{iv})}(u)\tet'(u) = O(\tet''(u))$, 
  $\sup\nolimits_{v\gt u}|\tet''(v)| =  O(\tet''(u))$, and  $\tet''(u+o(u)) =  O(\tet''(u))$.
  \ei
  Let us check (A1)--(A5).   (A1) is clear.
  (A2') and (A4'), which are equivalent to (A2) and, respectively, (A4) by Remark \ref{rem161},
  follow immediately from (i).
  
  To check (A3), note that $g(t) = t+ \td \tet(t)$, where $\td\tet(t) = -\tet(u)$ with $u = g(t)$. We further note that
  $g'(t) = 1+\td \tet'(t) = \frac1{f'(u)} = \frac1{1+\tet'(u)}$, which shows that $\lim_{t\to+\infty} \td \tet'(t)  = 0$ and immediately
  implies (A3). It remains to verify (A5). It is straightforward to compute
   \aa{
      \td \tet''(t) = -\frac{\tet''(u)}{(1+\tet'(u))^4} \quad \text{and} \quad 
   \td \tet'''(t) = \frac{4\tet''(u)}{(1+\tet'(u))^6} - \frac{\tet'''(u)}{(1+\tet'(u))^5},
   }
   which, by (iii), implies the first identity in (A5)-(a).
   Computing the fourth derivative of $\td \tet(t)$, it is straightforward to obtain  
   the second identity in (A5)-(a). The expression for  $\td \tet''(t)$ implies (A5)-(b)
   if we note that $t+o(t) = f(u+o(u))$.
  \subsection{Frequently used notation}
  \lb{notation}
  For the reader's convenience, we introduce a list of symbols that will be frequently used throughout the paper:
  
  $\te =   \gm +2+ \beta - \al$; $\hat \al = \frac{\al}{\beta+1}$; $\hat \te = \frac{\te}{\beta+1}$.
  
  $(u_\lm,\la)$ and $(u^*_{\lm^{\!*}},\la^*)$ denote regular and singular solutions to problem \rf{G}.
  
  $u$ and $u^*$ are rescaled, by \rf{ch-var}, regular and singular solutions that solve \rf{sing1}.
  
  $g'$, $g''$, $g'''$, $g^{(\rm{iv})}$ denote  derivatives of the respective orders w.r.t. their arguments.

\subsection{Singular solution in a neighborhood of the origin}
One of the main results of this work is the representation of a singular solution 
announced in Theorem \ref{thm1-}. The statement of this theorem holds, however,
under the more general assumptions (A1)--(A6).
\begin{thm}
\lb{thm1}
Assume (A1)--(A6). Then, the statement of Theorem \ref{thm1-} holds true.
\end{thm}
Let $\Gm(t) = F^{-1}(t)$; namely, $\Gm(t) = g(\ln t)$. The result of Theorem \ref{thm1}
can be rewritten as follows.
\begin{cor}
\lb{cor1}
Assume (A1)--(A6). Then,
there exists $\la^*>0$ such that problem \rf{G} possesses a singular radial solution of the form
\eqn{
u^*_{\lm^{\!*}}(r) = & \,\Gm(\widehat Z(\tau))
+O\big(\Gm'(\tau)\tau + \Gm''(\tau)\tau^2\big), \quad \text{where} \; \; \tau = \frac{\beta+1}{\la^*r^\te},\notag\\
\text{and} \quad  
& \widehat Z(\tau) = \te^{\beta+1}(\hat\al -1) \, \tau^{\beta+2} \big[\Gm '(\tau) + (\beta+1)
\ln \{\Gm'(\tau)\tau\}(\Gm'(\tau)+\Gm''(\tau)\tau)\big]^{\beta+1}. \notag
}
\end{cor}
\begin{proof}
Note that $g'(\ln \tau) = \Gm'(\tau) \tau$ and $g''(\ln \tau) = \Gm'(\tau) \tau + \Gm''(\tau) \tau^2$.
Therefore,
  \mm{
 g\big( \ln\big\{\te^{\beta+1}(\hat\al -1)\big\}  + \ln \tau 
+ (\beta+1)\ln \{g'(\ln\tau)+(\beta+1)g''(\ln\tau) \ln g'(\ln\tau)\}\big)\\
= g\big(\ln\{ \te^{\beta+1}(\hat\al -1)\tau (g'(\ln\tau)+(\beta+1)g''(\ln\tau) \ln g'(\ln\tau))^{\beta+1}\big\}\big)\\
= \Gm\big(\te^{\beta+1}(\hat\al -1) \, \tau^{\beta+2} \big\{\Gm'(\tau) + (\beta+1)
\ln\{\Gm'(\tau)\tau\}(\Gm'(\tau)+\Gm''(\tau)\tau)\big\}^{\beta+1}\big).
}
This completes the proof of the corollary.
  \end{proof}
  In Corollary \ref{cor2222}, we obtain asymptotic representations for singular solutions 
 when $e^{f(u)}$ is one of the nonlinearities considered in \cite{ghergu,kikuchi,miyamoto2015}. 
The respective representations from \cite{ghergu,kikuchi,miyamoto2015} agree with ours
 when $L$ is the radial Laplacian. 
\begin{cor}
\lb{cor2222}
Assume (A6) and consider the following three particular cases of $f$: $(a) \; f(u) = e^u;$ 
$(b)\; f(u) = u^p,$ $(p>\frac12);$ $(c)\; f(u) = u+ \tet(u),$ where $\tet(u)$ satisfies assumptions (i),(ii),(iii) from Example 3 in
Subsection \ref{examples} and such that $|\tet(u)| + |\tet'(u)|\lt C_0e^{-\dl u}$. Then,
\aa{
&(a) \; u^*_{\lm^{\!*}}\Big(\frac{r}{{\la^*}^{\sfrac1\te}}\Big) =\ln\Big\{ \te \ln\frac1r
+ \ln \frac{\hat\al -1}{\ln\frac{k}{r}} + (\beta+1)\ln\Big(1+ (\beta+1)\frac{\ln\big\{\te \ln\frac{k}{r}\big\}}{\te \ln\frac{k}{r}}\Big)\Big\}\\ &\hspace{9cm} +O\Big(\Big(\ln\frac{1}{r}\Big)^{-2}\Big);\\
&(b) \; u^*_{\lm^{\!*}}\Big(\frac{r}{{\la^*}^{\sfrac1\te}}\Big) = \Big( \te \ln\frac1r + (\beta+1)\Big(\frac1p-1\Big) \ln \ln\frac{k}{r}
+ \ln \frac{(\al-\beta-1) \te^{\frac{\beta+1}{p}}}{p^{\beta+1}} \\
&\phantom{(b) \; u^*_{\lm^{\!*}}\Big(\frac{r}{{\la^*}^{\sfrac1\te}}\Big)}
+(\beta+1)\ln\Big\{1+\Big(\frac1p-1\Big)^2\frac{\ln\tau}{\tau} -\Big(\frac1p-1\Big)\frac{\ln p}{\tau}\Big\}\Big)^\frac1p
+O\Big(\Big(\ln\frac1r\Big)^{\frac1p-2}\Big);
\\
&(c)\; u^*_{\lm^{\!*}}(r)  = \ln \big\{ \te^{\beta+1}(\al-\beta-1)\} - \ln \la^*-\te \ln r   + O(r^{\te \dl}).
}
Above, $k = (1+\beta)^\frac1\te$, $\tau = \te \ln\frac{k}{r}$.
\end{cor}
\begin{proof}
(a) and (b) follow immediately from representation \rf{as1}.
 To show (c), we note that
   in a neighborhood of infinity, by (iii), 
   $\tet''(u)$ either equals zero  or $\frac{\tet'''(u)}{\tet''(u)} = O(1)$. By L'Hopital's rule
   and (i), $\lim {\rm sup}_{u\to +\infty} \big|\frac{\tet''(u)}{\ln\{1+\tet'(u)\}}\big|$ is finite, 
   which means that $|\tet''(u)| \lt C_1 e^{-\dl u}$ for some constant $C_1>0$. Furthermore,
   the representations for $\td \tet$, $\td \tet'$, and $\td \tet''$ imply that in a neighborhood of infinity,
   $|\td \tet(t)|  + |\td \tet'(t)| + |\td \tet''(t)|\lt C_2e^{-\dl t}$ for some constant $C_2>0$. 
   Representation (c) follows now from \rf{as1}.
   \end{proof}
   \begin{rem}
   \rm
   In example (b), we deal with the same nonlinear term as in \cite{kikuchi}. 
 We would like to emphasize the following fact. 
   In \cite{kikuchi}, the asymptotic representation for the singular solution (Theorem 1.1), in the case $0<p<1$, 
   is given up to a 
   rest term (not computed explicitly) which does not tend to zero.  Unlike \cite{kikuchi}, 
   our asymptotic representation includes
   additional explicit terms,
  tending to infinity, so our rest term is of a higher order than in \cite{kikuchi} and goes to zero.
 \end{rem}
The proof of Theorem \ref{thm1} is divided into the steps outlined below.  
\subsubsection{Equivalent problem in a neighborhood of infinity}
\lb{s223}
It is convenient  to rewrite problem \rf{G}
by using the change of variable $t = \ln(\frac{\kappa}{r})$, where $\kappa>0$ is a constant to be specified later.
 First, we note that  the equation in \rf{G} can be rewritten as follows:
 \aa{
u''(r)|u'(r)|^\beta + \hat \al \,\frac{u'(r)|u'(r)|^\beta}{r} + \frac{\la\, r^{\gm-\al}\, e^{f(u)}}{(\beta+1)} = 0.
}
Let us show that for a singular solution to \rf{G}, it holds that $(u^*_{\lm^{\!*}})'(r)<0$ for all $r>0$.
Note that if $\beta > 0$, from the above equation, it follows that a singular solution cannot have local maximum or minimum points,
since in those points  $u'(r) = 0$.
Suppose $\beta = 0$. Then, every local extremum point, is necessarily a local maximum point, which contradicts 
to the fact that $\lim_{r\to 0+} u^*_{\lm^{\!*}}(r) = +\infty$. 

By doing the change of variable $t = \ln(\frac{\kappa}{r})$
 and defining  $v(t) = u(r)$, we transform \rf{G} to the following problem:
\eqq{
v'' - (\hat \al -1) v' + \frac{\la  \kappa^\te  e^{-\te t} e^{f(v)}}{(\beta+1)|v'|^\beta} = 0, \quad t\in (\ln \kappa,+\infty), \\
v(\ln \kappa) = 0.
}
By choosing $\kappa$ in such a way that $\frac{\la \, \kappa^\te}{\beta+1} = 1$, 
we arrive at equation \rf{eq-g1} with the boundary condition 
$v(\ln \kappa) = 0$,  where $\kappa =  \big(\frac{\beta+1}{\la}\big)^\frac1\te$.
This problem  is equivalent to \rf{G}, but given in a neighborhood of infinity.
\begin{pro} 
\lb{lem31}
Assume (A1)--(A6). 
Then, there exists $T>0$ such that on $[T,+\infty)$, equation \rf{eq-g1}
possesses a singular solution of the form \rf{sol-g1}, where
$\eta(t)=  O(g''(\te t))$ is a $\C^2$-function.
\end{pro}
\begin{rem}
\rm
\lb{rem17}
Remark that by Lemma \ref{lem223} and Remark \ref{rem16}, $\lim_{t\to+\infty}\frac{\ln g'(\te t)}{t} = 0$.
Therefore, $(\ffi_1+\ffi_2)(t) = o(t)$.
\end{rem}
\begin{proof}
\textit{Step 1. Transformation to a fixed-point problem.} 
The idea is to obtain a semilinear equation with respect to $\eta$ by substituting $v^*= g + \eta$ into \rf{eq-g1} 
and showing that the difference of singular terms that occur on the both sides of the resulting equation 
make a regular function. 
More specifically, we are going to show that
the original problem is equivalent to a couple
semilinear equations with respect to $\eta$ and $\zeta=\eta'$
\eq{
\lb{linear1}
\eta''(t) - (\al-\beta-1)\, \eta'(t) + \te (\hat\al-1)\,\eta(t)  = \Phi(t,\eta, \zeta),\\
\zeta'(t) -  (\al-\beta-1)\, \zeta(t) = \Phi(t,\eta,\zeta) -  \te(\hat\al-1)\, \eta(t),
}
where the function $\Phi$ is expected to be regular. For simplicity of notation, define
\aa{
\ffi = \ffi_1 + \ffi_2,
}
where $\ffi_1$ and $\ffi_2$ are defined by \rf{ffi}.
We start by transforming $\exp\{ - \te t +f(v^*)\} = \exp\{-\te t + f(g(\te t+\ffi) + \eta)\}$,
 taking into account that $g=f^{-1}$. It holds that
\aa{
 f(g(\te t+\ffi) + \eta) = \te t+\ffi + f'(g(\te t+\ffi)+\tet(t,\eta) \eta)\eta,
}
where $\tet(t,\eta)$ takes values between $0$ and $1$.
Let us transform the term $f'(g(\te t+\ffi)+\tet(t,\eta) \eta)$. 
Define 
\aa{
\xi(t) = (\te t+\ffi)\Big(\frac{f\big(g(\te t+\ffi) + \tet(t,\eta) \eta\big)}{\te t+\ffi} -1\Big).
}
Then, $g(\te t+\ffi)+\tet(t,\eta) \eta = g(\te t+\ffi + \xi)$. By (A2) and Remark \ref{rem17}, $\xi = o(\te t+\ffi)
= o(\te t + o(t)) = o(t)$.
Therefore,
\aaa{
\lb{eq5}
f(g(\te t+\ffi) + \eta) = \te t+\ffi+ \frac{\eta}{g'(\te t+\ffi+\xi)}
=  \te t+\ffi+ \frac{\eta}{g'(\te t+o(t))}.
}
Therefore,
\aaa{
\lb{rhs1}
\frac{\exp\{ - \te t +f(v^*)\}}{|{v^*}'(t)|^\beta} = \frac{e^{\ffi_1}}{{|{v^*}'(t)|^\beta}}\, e^{\ffi_2+\frac{\eta}{g'(\te t+o(t))}}.
}
Let us transform the first factor on the right-hand side. Defining $\tet_1(t) = \te^{-1}\ffi'(t)$, we obtain
\mmm{
\lb{t5}
\frac{e^{\ffi_1}}{|{v^*}'(t)|^\beta} = \frac{\te^\beta\big(g'(\te t) + (\beta+1)g''(\te t) \ln g'(\te t)\big)^\beta}
{\te^\beta |g'(\te t + \ffi)(1+ \tet_1)+\te^{-1}\eta'|^\beta} \\ = \frac1{ \big | 1 + \frac{\eta'}{\te g'(\te t+\ffi)(1+\tet_1)}\big |^\beta}
\, \frac1{(1+\tet_1)^\beta} \, \frac1{(1+\tet_2)^\beta},
}
where $\tet_2$ is defined through the formula
 \aa{
 g'(\te t+ \ffi) = (g'(\te t) + (\beta+1)g''(\te t) \ln g'(\te t))(1+\tet_2(t)).
 }
 Let us obtain expressions for $\tet_1$ and $\tet_2$ convenient for future computations.  By (A4) and (A5), 
 \mmm{
\lb{tet1}
\tet_1(t) = \te^{-1} \ffi'(t) = (\beta+1)\, \frac{g''(\te t) + (\beta+1)g'''(\te t) \ln g'(\te t) + (\beta+1)\frac{g''(\te t)^2}{g'(\te t)}}
{g'(\te t) + (\beta+1)g''(\te t) \ln g'(\te t)} \\ = O\Big(\frac{g''(\te t)}{g'(\te t)}\Big).}
Since 
\aaa{
\lb{pro-ffi}
\ffi(t) = (\beta+1)\ln g'(\te t) + O(1),
}
by using the expansion of $g'(\te t+\ffi)$ by Taylor's formula
around $\te t$, we obtain 
\mmm{
\lb{tet2}
\tet_2(t) = \frac{g'(\te t+ \ffi) - g'(\te t) - (\beta+1)g''(\te t) \ln g'(\te t)}{g'(\te t) + (\beta+1)g''(\te t) \ln g'(\te t)}\\
=\frac{g''(\te t) O(1)  + \int_0^1(1-r) g'''(\te t + r \ffi) dr \, \ffi^2}{g'(\te t) + (\beta+1)g''(\te t) \ln g'(\te t)} = O\Big(\frac{g''(\te t)}{g'(\te t)}\Big).
}
The last identity holds by (A4), (A5), and the following argument. First of all, we note that 
by Taylor's formula  around $\te t$ and \rf{pro-ffi},
\aa{
\ln g'(\te t+\ffi) = \ln g'(\te t) + \ln\Big(1+O\Big(\frac{g''(\te t)}{g'(\te t)}\Big) + O\Big(\frac{g''(\te t)}{g'(\te t)}\ln g'(\te t)\Big) \Big)
= \ln g'(\te t)  + o(1),
}
which,  by Remark \ref{rem17},  implies that 
\aa{
g'''(\te t + r \ffi) \ffi^2 = O(g''(\te t+\ffi)) = O(g''(\te t)).
}

Note that since we are interested in solving equation \rf{eq-g1}
in a neighborhood of infinity, the expressions $\frac1{(1+\tet_1)^\beta}$ and $\frac1{(1+\tet_2)^\beta}$
are well-defined.

Next, by the results of Subsection \ref{s223}, for any singular solution to \rf{G}, it holds that
$(u^*_{\lm^{\!*}})'(r)<0$ for all $r>0$. This implies that
${v^*}'(t)>0$ for all $t$. Then, by Taylor's formula, \rf{t5} can be transformed as follows:
\aa{
\frac{e^{\ffi_1}}{{v^*}'(t)^\beta}  =  \Big(1 -\frac{\beta \eta'}{\te g'(\te t)}+ o(1)\, \frac{\eta'}{g'(\te t)}
+ \frac{{\eta'}^2}{g'(\te t)^2}\, \sg(t, \eta') \Big)\Big(1+O\Big(\frac{g''(\te t)}{g'(\te t)}\Big)\Big),
}
where 
 \aaa{
 \lb{sg1}
\sg(t, \eta')=
 O(1) \int_0^1 \frac{(1-r)}{(1+rx)^{\beta+2}}\,dr  \quad \text{with}   \; \; 
 x =  \frac{\eta'}{\te g'(\te t+\ffi)(1+\tet_1)}.
 }
Recalling that $\frac{g''(\te t)}{g'(\te t)} \to 0$ as $t\to +\infty$ (see Remark \ref{rem16}),
the last expression can  be rewritten as 
\aaa{
\lb{t6}
\frac{e^{\ffi_1}}{{v^*}'(t)^\beta}  = 
1 -\frac{\beta \eta'}{\te g'(\te t)} +o(1)\frac{\eta'}{g'(\te t)}
+  O(1)\,\frac{{\eta'}^2}{g'(\te t)^2}\,\sg(t, \eta')  + O\Big(\frac{g''(\te t)}{g'(\te t)}\Big).
}
Taking into account that $e^{\ffi_2} = \te(\hat\al-1)g'(\te t)(1+o(1))$, 
by \rf{rhs1}, \rf{tet1}, and \rf{t6}, expression \rf{rhs1} can be transformed as follows:
\mm{
\frac{\exp\{ - \te t +f(v^*)\}}{{v^*}'(t)^\beta} = e^{ \frac{\eta}{g'(\te t+o(t))}}e^{\ffi_2} \frac{e^{\ffi_1}}{{v^*}'(t)^\beta} =
e^{\frac{\eta}{g'(\te t+o(t))}}\Big[e^{\ffi_2} - \beta(\hat\al-1)\eta' + o(1)\eta' \\ + O(1)\, \frac{{\eta'}^2}{g'(\te t)}\sg(t,\eta')
+ O\big(g''(\te t)\big)\Big]
= e^{\ffi_2+ \frac{\eta}{g'(\te t+o(t))}} - \beta(\hat \al -1) \eta' + F_1(t,\eta,\eta') + F_2(t,\eta)\\
  = e^{\ffi_2}\Big(e^{\frac{\eta}{g'(\te t+o(t))}} -  \frac{\eta}{g'(\te t+o(t))} -1\Big)
+ \frac{e^{\ffi_2}\eta}{g'(\te t+o(t))} + e^{\ffi_2} 
- \beta(\hat \al -1) \eta' \\ + F_1(t,\eta,\eta') + F_2(t,\eta),
}
where
\aa{
&F_1(t,\eta,\eta') = \Big[o(1)\eta' +
O(1)\, \frac{{\eta'}^2}{g'(\te t)}  \sg(t, \eta') \Big]e^{\frac{\eta}{g'(\te t+o(t))}}\\
& \hspace{1.6cm} -\beta(\hat\al-1)\eta'\big(e^{\frac{\eta}{g'(\te t+o(t))}} -1\big)
+  O\big(g''(\te t)\big) (e^{\frac{\eta}{g'(\te t+o(t))}}-1);\\
&F_2(t,\eta) = O\big(g''(\te t)\big).
}

Now we are ready to evaluate differences of suitable singular terms  that would result in regular functions.
More specifically, we aim to obtain \rf{linear1} from \rf{eq-g1} by substituting $v^*(t) = g(\te t + \ffi) + \eta(t)$ into \rf{eq-g1}.
We have
\aaa{
\lb{t1}
\frac{e^{\ffi_2}\eta}{g'(\te t+o(t))}- \te (\hat\al-1)\, \eta =  \te  (\hat\al-1)
\Big(\frac{g'(\te t) (1+o(1))}{g'(\te t+o(t))} - 1\Big) \eta.
}
Further, by the definition of $\ffi_2$, representation \rf{pro-ffi}, and Taylor's expansion for  $g'(\te t+\ffi)$ around $\te t$ (in the same
way as we applied it in the numerator of \rf{tet2}),
\mmm{
\lb{t2}
e^{\ffi_2} -  (\hat\al-1) \frac{d}{dt}\, g(\te t +\ffi) = \te  (\hat\al-1) \big( g'(\te t) + (\beta+1)g''(\te t)\ln g'(\te t)\\
- g'(\te t+\ffi)(1 +\tet_1)\big) = O(g''(\te t)).
}
 In \rf{t2}, we used (A5) and representation \rf{tet1} for $\tet_1$.
Next,
\aaa{
\lb{t3}
\frac{d^2}{dt^2}\, g(\te t+\ffi) =  g''(\te t+\ffi)(\te +\ffi')^2 +  g'(\te t+\ffi)\ffi''  = O(g''(\te t)).
}
This holds by (A3), (A5), formula \rf{tet1} for $\ffi'$, and the following expression for $\ffi''$:
\aa{
\ffi''(t) = \te^2(\beta+1)  \, \frac{g'''(\te t) + g^{(\rm{iv})}(\te t)\ln g'(\te t) + \frac{3g'''(\te t)g''(\te t)}{g'(\te t)} -\frac{g''(\te t)^3}{g'(\te t)^2}}
{g'(\te t) + g''(\te t) \ln g'(\te t)} - (\ffi')^2.
}
Thus, we obtain \rf{linear1} with 
\aaa{
\lb{phi}
\Phi(t,\eta,\zeta) = -\big(F_1(t,\eta,\zeta) + F_2(t,\eta) + F_3(t,\eta) + F_4(t) +F_5(t,\eta)\big),
}
where $\eta'$ is substituted by $\zeta$ on the right-hand side of the first equation of \rf{linear1}
and in the second equation.
Furthermore, $F_3$ is defined by \rf{t1},  $F_4(t) = O(g''(\te t))$ comes due to \rf{t2} and \rf{t3},
and $F_5$ is defined as follows:
\mmm{
\lb{t4}
F_5(t,\eta) = e^{\ffi_2}\Big(e^{\frac{\eta}{g'(\te t+o(t))}} -  \frac{\eta}{g'(\te t+o(t))} -1\Big)\\
= \te (\hat\al-1) \int_0^1 \!(1-r) \, e^{\frac{r\eta}{g'(\te t+o(t))}} dr \;  
\frac{g'(\te t)O(1)}{g'(\te t+o(t))^2} \, \eta^2.
}
We will be searching for pairs $(\eta,\zeta)$ that solve \rf{linear1} in a certain metric space which we denote by $X_T$.
If $g''(t)$ is not an identical zero in a neighborhood of infinity, we define 
\aa{
X_T = \{x\in \C[T,\infty): x(t) = O(g''(\te t)) \; \text{as} \; t\to+\infty\}.
}
Note that under (A5)-(b), either $g''(\te t)$ does not  take zero values  in a neighborhood of  infinity (case (i)), or
 identically equals zero in a neighborhood of  infinity  (case (ii)). In the case (i), for sufficiently large $T$, the
 following norm in $X_T$ is well-defined:
\aa{
\|x\|_{X_T} = \sup_{t\gt T}\{|g''(\te t)|^{-1}|x(t)|\}.
}
In the case (ii), we define $X_T = \C_b([T,+\infty))$
with the supremum norm. Here, $\C_b$ stands for the space of bounded continuous functions.

In $X_T$, \rf{linear1} uniquely transforms to the following  couple of fixed-point equations:
\aaa{
\lb{fpm}
\eta = \Psi(\eta, \zeta), \qquad \zeta = {\hat\Psi}(\eta,\zeta).
}
Since $\hat\al>1$, the explicit formula for $\hat\Psi$ is the following:
\aaa{
\lb{tpsi1}
{\hat\Psi}(\eta,\zeta)(t) =  \int_t^{+\infty}  e^{2\dl (t-s)}
\big(\Phi(s,\eta,\zeta) -\te(\hat\al-1) \, \eta(s)\big) ds,
}
where $\dl =  \frac{\al-\beta-1}2$.
The explicit formula for $\Psi$ depends on the sign of the discriminant $D= 4(\dl^2 - \te(\hat\al-1))$
of the corresponding characteristic equation.
In the case $D<0$, we have
\aaa{
\lb{case1}
\Psi(\eta,\zeta)(t) = \frac1{\mu}\, \int_t^{+\infty} 
e^{\dl (t-s)} \sin\big(\mu(s-t)\big) \Phi(s,\eta, \zeta)\, ds, \quad \text{where} \; \;\mu =  \frac{\sqrt{|D|}}2.
}
In the case $D>0$, we have
\aaa{
\lb{case2}
\Psi(\eta,\zeta)(t) = \frac1{2\mu}\,  \int_t^{+\infty} 
e^{(\dl - \mu) (t-s)}\big(1  -  e^{2\mu(t-s)}\big)
 \Phi(s,\eta, \zeta)\, ds.
}
Finally, if $D=0$, we have
\aaa{
\lb{case3}
\Psi(\eta,\zeta)(t) =\int_t^{+\infty} e^{\dl (t-s)} (s-t) \Phi(s,\eta,\zeta)\, ds.
}

\textit{Step 2. Existence of a fixed point.} 
Let us show that we can find constants $M, \hat M >0$ and a number $T>0$ such that
the map $(\Psi,{\hat\Psi})$ acts from $\Sg$ to $\Sg$,
where the latter is defined as a complete metric space of the following form:
 \aa{
 \Sg = \Sg(M,\hat M, T) = \{(\eta,\zeta) \in X_T \x X_T: \|\eta\|_{X_T} \lt M \; \text{and} \; \|\zeta\|_{X_T}\lt \hat M \}.
 }
 We define the metric in $\Sg$ as follows:
 \aa{
 \rho\{(\eta_1,\zeta_1), (\eta_2,\zeta_2)\} = \|\eta_1-\eta_2\|_{X_T} + \|\zeta_1-\zeta_2\|_{X_T}.
 }
Moreover, we show that $T$ can be chosen in such a way that 
 $(\Psi,{\hat\Psi})$ is a contraction map $\Sg\to \Sg$. 

The integral operator applied to $\Phi$ and defined by \rf{case1}, \rf{case2},
or \rf{case3} splits into five terms  (which we denote by $I_1$, $I_2$, $I_3$, $I_4$, and $I_5$),
corresponding to the decomposition of $\Phi$ by formula \rf{phi}. 
We evaluate each of these terms. First, we compute the integrals
\aa{
\mathsmaller{\frac1{\mu}} \int_t^{+\infty} \hspace{-2mm} e^{\dl (t-s)} ds = \mathsmaller{\frac1{\mu\dl}}; \; \; 
 \mathsmaller{\frac1{2\mu}} \int_t^{+\infty}  \hspace{-2mm} e^{(\dl-\mu)(t-s)} ds = \mathsmaller{ \frac1{2\mu(\dl-\mu)}}; \; \; 
\int_t^{+\infty} \hspace{-2mm} e^{\dl (t-s)} (s-t) ds = \mathsmaller{ \frac1{\dl^2}};
}
and define
\aa{
\varkappa = \max\big\{ \mathsmaller{\frac1{2\dl}, \frac1{2\mu(\dl-\mu)}, \frac1{\mu\dl}, \frac1{\dl^2}}\big\}.
}
The expression for $\varkappa$ involves the values of the above integrals and 
the term $\frac1{2\dl}$  which appears from the integration of $e^{2\dl (t-s)}$ in \rf{tpsi1}.
We make analysis separately in the case (i), when $g''(\te t)\ne 0$ in a neighborhood  of infinity,  and in the case (ii),
when $g''(\te t)$ identically equals zero  in a neighborhood  of infinity.
We start from the case (i). Introduce constants
$M_1$, $M_2$, and a number $T_1>0$ such that
\aaa{
\lb{M1}
&\sup_{s\gt t}|g''(\te s)| \lt M_1|g''(\te t)|, \quad t\gt T_1;\\
&|F_i(t)| \lt M_2 |g''(\te t)|, \quad i=2,4, \quad t\gt T_1. \notag
}
Let us fix the metric space for the fixed point argument.
Take $M > 5  \varkappa M_1 M_2$ and  $\hat M = \frac{\te M M_1}{1+ \beta}  + M$.
Choose $T>T_1$ in such a way that if the a priori estimates $|\eta(t)|\lt M|g''(\te t)|$ and $|\zeta(t)|< \hat M |g''(\te t)|$ hold,
then,
\aa{
|F_i(t)| \lt M_2 |g''(\te t)|, \quad i=1,3,5, \quad t\gt T.
}
 We can always do this since for $\eta$ and $\zeta$ with the above property, $F_1$, $F_3$, and $F_5$ are $o(g''(t))$.
Now take $(\eta,\zeta)\in \Sg(M,\hat M, T)$ and note that each of the terms $I_i$, $i=1,2,3,4,5$,
 becomes smaller than $\varkappa M_1 M_2$, which is smaller than $\frac{M}5$.
 Thus, we have proved that $|\Psi(\eta,\zeta)(t)| \lt M |g''(t)|$. 
 Next, we represent ${\hat\Psi}$ as $\Psi_1 + \Psi_2$, where 
 \eqn{
 \lb{ps23}
 &\Psi_1(\eta, \zeta)(t)  = \int_t^{+\infty}  e^{2\dl (t-s)} \Phi(s,\eta(s),\zeta(s)) ds,\\
  &\Psi_2(\eta)(t)  =  -\te(\hat\al-1) \int_t^{+\infty}  e^{2\dl (t-s)} \eta(s) ds.
 }
From the above argument, we obtain that $|\Psi_1(\eta,\zeta)(t)| \lt M |g''(t)|$. Furthermore,
$|\Psi_2(\eta)(t)| \lt  \frac{\te M M_1}{1+ \beta} |g''(t)|$.
 Therefore, by the choice of $\hat M$, we obtain that $|{\hat\Psi}(\eta,\eta')(t)| \lt \hat M |g''(t)|$.
Thus, we proved that $\Psi(\eta,\zeta)$ is a map $\Sg \to \Sg$.
 
Let us show that if the number $T$ is sufficiently large, then  $\Psi:\Sg\to\Sg$  is a contraction map.
Let $\chi(s-t)$ denote one of the factors $\sin\{\mu(s-t)\}$, $1  -  e^{2\mu(t-s)}$, or  $s-t$
which occur in the integrals \rf{case1}, \rf{case2}, and \rf{case3}.
By (A5)-(b) and the integrability of the exponential factor, multiplied by $\chi$,
it suffices to observe that
for two elements $(\eta_1, \zeta_1)$ and  $(\eta_2, \zeta_2)$ from $\Sg$, and for each $t\in (T,+\infty)$,
  \eqn{
  \lb{Fi}
  &|F_1(t,\eta_1,\zeta_1) - F_1(t,\eta_2,\zeta_2)| \lt \epsilon(t) (|\eta_1-\eta_2| + |\zeta_1-\zeta_2|), \\
  &|F_i(t,\eta_1) - F_i(t,\eta_2)| \lt \epsilon(t) |\eta_1-\eta_2|, \quad i= 3,5,
  }
where $\epsilon$  it is a positive function tending to $0$ as $t\to +\infty$.
Inequalities \rf{Fi}
follow immediately from the expressions for $F_i$, $i=1,3,5$, 
if we take into account the estimates  $|\eta_i(t)| \lt M |g''(t)|$, $|\zeta_i(t)| \lt \hat M |g''(t)|$, $i=1,2$,
Remark \ref{rem16}, and assumption (A3).
We also observe that, by \rf{sg1},
$\sg_1(t, \zeta)$ is Lipschitz in its second argument for sufficiently large $t$.
By \rf{M1} and  \rf{Fi}, 
\mm{
\sup_{t\gt T} |g''(\te t)|^{-1} |\Psi(\eta_1,\zeta_1)(t) - \Psi(\eta_2,\zeta_2)(t)| \\ \lt 
5\sup_{t\gt T}\Big\{|g''(\te t)|^{-1}\int_t^{+\infty} e^{\dl(t-s)}\chi(t-s) \epsilon(s)
\big(|\eta_1(s)-\eta_2(s)| + |\zeta_1(s)-\zeta_2(s)| \big)ds\Big\}\\
\lt 5M_1 \sup_{t\gt T}\epsilon(t) \int_T^{+\infty} e^{\dl(t-s)}\chi(t-s)
|g''(\te s)|^{-1}\big(|\eta_1(s)-\eta_2(s)| + |\zeta_1(s)-\zeta_2(s)| \big) ds.
}
Thus, by choosing $T$ sufficiently large, we can make the factor in front of the integral small, so that
there exists $r>0$ sufficiently small (whose exact bound will be specified later)
such that
\aaa{
\lb{ev1}
\|\Psi(\eta_1,\zeta_1) - \Psi(\eta_2,\zeta_2)\|_{X_T} \lt r (\|\eta_1-\eta_2\|_{X_T} + \|\zeta_1-\zeta_2\|_{X_T}).
}
By \rf{tpsi1},
\mm{
\|{\hat\Psi}(\eta_1,\zeta_1) - {\hat\Psi}(\eta_2,\zeta_2)\|_{X_T} \lt \|\Psi_1(\eta_1,\zeta_1) - \Psi_1(\eta_2,\zeta_2)\|_{X_T} 
+ \|\Psi_2(\eta_1) - \Psi_2(\eta_2)\|_{X_T}, 
}
where $\Psi_1$ and $\Psi_2$ are given by \rf{ps23}. The first term in the above identity is evaluated exactly as in \rf{ev1}. 
For the second term, we have
\aa{
\|\Psi_2(\eta_1) - \Psi_2(\eta_2)\|_{X_T} \lt R \|\eta_1-\eta_2\|_{X_T} \quad \text{with} \; \; R=\frac{\te M_1}{1+\beta} 
}
(recall that $2\dl =  \al-\beta-1$). Therefore,
\aa{
\|{\hat\Psi}(\eta_1,\zeta_1) - {\hat\Psi}(\eta_2,\zeta_2)\|_{X_T}  \lt  (R+r) \|\eta_1-\eta_2\|_{X_T} + r  \|\zeta_1-\zeta_2\|_{X_T}.
}
Now let $(\Psi^{(2)},{\hat\Psi^{(2)}}) = (\Psi,{\hat\Psi})\circ (\Psi,{\hat\Psi})$. We show that one can find 
a number $r\in (0,1)$ such that $(\Psi^{(2)},{\hat\Psi^{(2)}})$  is a contraction map.
We have
\eqm{
&\|\Psi^{(2)}(\eta_1,\zeta_1) - \Psi^{(2)}(\eta_2,\zeta_2)\|_{X_T} \lt r(2r+R) \|\eta_1-\eta_2\|_{X_T}  + 2r^2  \|\zeta_1-\zeta_2\|_{X_T}, \\
&\|{\hat\Psi^{(2)}}(\eta_1,\zeta_1) - {\hat\Psi^{(2)}}(\eta_2,\zeta_2)\|_{X_T} \lt 2r(r+R) \|\eta_1-\eta_2\|_{X_T} + r(2r+R)  \|\zeta_1-\zeta_2\|_{X_T}.
}
Therefore,
\mm{
\|\Psi^{(2)}(\eta_1,\zeta_1) - \Psi^{(2)}(\eta_2,\zeta_2)\|_{X_T}  + \|{\hat\Psi^{(2)}}(\eta_1,\zeta_1) - {\hat\Psi^{(2)}}(\eta_2,\zeta_2)\|_{X_T}\\
\lt  (4r^2 + 3Rr) ( \|\eta_1-\eta_2\|_{X_T} +   \|\zeta_1-\zeta_2\|_{X_T}).
}
Picking $r>0$ in such a way that $4r^2 + 3Rr <1$, we obtain that $(\Psi^{(2)},{\hat\Psi^{(2)}})$ is a contraction 
map $\Sg \to \Sg$. Let $(\eta,\zeta)$ be its unique fixed point. Then, $(\eta,\zeta)$  is also a fixed point of $(\Psi,{\hat\Psi})$, while
$\eta$ determines solution \rf{sol-g1} with the prescribed singular part given by the first term. To see this, it suffices
to note that the solution $(\eta,\zeta)$, that we just found, satisfies the property $\eta' = \zeta$. Indeed, $\bar \zeta=\eta'$ also satisfies the second
equation in \rf{linear1} which implies that
\aa{
\eta'(t) = \hat \Psi(\eta,\zeta)(t) + C e^{(\al-\beta-1) t},
}
where $C$ is a constant. On the other hand, \rf{case1}, \rf{case2}, or \rf{case3} imply that
\aa{
\eta' = \frac1{\gm}\, \int_t^{+\infty} 
\frac{d}{dt} \big\{ e^{\dl (t-s)} \chi(t-s)\big\} \Phi(s,\eta, \zeta)\, ds,
}
which implies that $\eta'(t)= o(1)$.
Therefore, $C=0$ and $\zeta = \eta'$. 

Thus, we obtained a singular solution $v^*(t) = g(\te t + \ffi(t)) + \eta(t)$ to \rf{eq-g1}.

Consider now the case (ii), i.e., $g''(\te t) = 0$ on $[\hat T,+\infty)$ for some $\hat T>0$.
In this case, the terms $F_2$ and $F_4$
are equal to zero. 
The non-zero terms are $F_1$, $F_3$, and $F_5$.
Since $F_1(t,0,0) = F_3(t,0,0) =  F_5(t,0,0) = 0$, $\eta= 0$ is a solution to 
\aa{
\eta''(t) - (\al-\beta-1)\, \eta'(t) + \te (\hat\al-1)\,\eta(t)  = \Phi(t,\eta, \eta').
}
Hence, in the case (ii),  $v^*(t) = g(\te t + \ffi(t))$ is an exact singular solution to \rf{eq-g1}. 
\end{proof}
Now we extend the singular solution to \rf{eq-g1}, that we constructed in Proposition \ref{lem31}, to a solution 
to the same equation on the entire real line.
\begin{pro}
\lb{pro99}
In the assumptions of Proposition \ref{lem31}, the solution $v^*$ can be extended to 
a solution of \rf{eq-g1} on $(-\infty,+\infty)$.
Moreover, the extended solution $v^*$ is strictly increasing and 
there exists a number $T^*\in (-\infty,T)$ such that $v^*(T^*) = 0$.
\end{pro}
\begin{proof}
By Proposition \ref{lem31}, there exists a singular solution $v^*$  to \rf{eq-g1} on some interval $[T,+\infty)$.
Equation \rf{eq-g1} can be rewritten as follows:
\aaa{
\lb{32}
\big(e^{(\beta+1-\al)t} |{v}'(t)|^\beta {v}'(t)\big)' = -(\beta+1) e^{-(\gm+1) t} e^{f(v)}.
}
Substituting $v^*$ into \rf{32} and
integrating it from $t$ to $+\infty$, we obtain 
\aaa{
\lb{in4}
e^{(\beta+1-\al)t} |{v^*}'(t)|^\beta {v^*}'(t) = (\beta+1) \int_t^{+\infty} e^{- (\gm+1) s}  e^{f(v^*)} ds.
}
Indeed, since $\lim_{t\to+\infty} g''(t) = 0$, we conclude that $g'(t)$ has at most linear growth.
Furthermore, since ${v^*}'(t) = g'(\te t+\ffi)(\te + \ffi') + \eta'$, where $\ffi=o(t)$, $\ffi' =O\big(\frac{g''(\te t)}{g'(\te t)}\big)$ (see formula \rf{tet1})
and $\eta' = O(g''(\te t))$, we obtain that 
${v^*}'(t)$ also has at most linear growth on $[T,+\infty)$. 
Since $\beta+1-\al<0$,
the expression on left-hand side of \rf{in4} goes to zero as $t\to +\infty$, and hence, \rf{in4} holds.
Recall that ${v^*}'(t)>0$ (as we showed in the proof of Proposition \ref{lem31})
for any extension of $v^*$. 
Let $w(t) = e^{(\beta+1-\al)t} {v^*}'(t)^{\beta+1}$. We then obtain the following system with respect to $w$ and $v^*$:
\eq{
\lb{sys4}
w'(t) = -(\beta+1) e^{-(\gm+1) t} e^{f(v^*)},\\
{v^*}'(t) = e^{(\hat\al -1)t} w^\frac1{\beta+1}.
}
Pick $S>T$ and define $c_0 = v^*(S)$, $c_1 = w(S)$. Furthermore, we note that on $(-\infty,S]$, 
$c_1 \lt w(t) \lt  c_1 + c_2 \, e^{-(\gm+1)t}= \mf c_t$ for some constant $c_2>0$.
Define $\nu =  \ind_{(-\infty, c_0+\eps]} \ast \rho_\eps$
and $\sig_t =  \ind_{[c_1-\eps,\mf c_t+\eps]}\ast \rho_\eps$, 
where $\rho_\eps$, $\eps<c_1$, is a standard mollifier supported
on the ball of radius $\eps$. Note that $\nu(\fdot) = 1$ on $(-\infty,c_0]$ and $\sig_t(\fdot) = 1$ on $[c_1, \mf c_t]$.
Instead of \rf{sys4}, consider  the system of first-order ODEs on $(-\infty,S]$
\eq{
\lb{30}
w'(t) = -(\beta+1) e^{-(\gm+1)t} e^{\nu(v^*)f(v^*)},\\
{v^*}'(t) = e^{(\hat\al -1)t} (\sig_t(w) w)^\frac1{\beta+1}.
}
On $[T,S]$, $(v^*(t),w(t))$ is also a solution to \rf{30}, since over this interval $\nu(v^*(t)) = 1$ and $\sig_t(w(t)) =1$.
Next, since $e^{\nu(v)f(v)}\lt e^{f(c_0+2\eps)}$ and  $0<(\sig_t(w) w)^\frac1{\beta+1} \lt (\mf c_t + 2\eps)^\frac1{\beta+1}$, 
one can extend $(v^*,w)$ to $(-\infty, S]$
(see, e.g., \cite{filippov}, Chapter 2, \S 6),
obtaining by this a solution $(\td v,\td w)$ to \rf{30} which coincides with $(v^*,w)$ on $[T,S]$. 
Since on $(-\infty,S]$, $\nu(\td v) = 1$ and $\sig_t(\td w) = 1$,
 $(\td v, \td w)$ is also a solution to \rf{32}, and  therefore, to \rf{eq-g1}. 
We then use the same notation for this  extended solution, i.e., we write $(v^*,w)$ instead of $(\td v,\td w)$.
There are two possible situations: either $v^*$ is strictly 
positive over $(-\infty,S]$, or there exists $T^*\in\Rnu$ such that $v^*>0$ on $(T^*,S]$ and $v^*(T^*) = 0$.
Let us show that the first situation  cannot be realized. 
If  $v^*$ is strictly 
positive over $(-\infty,S]$, then there exists a finite limit
$L = \lim_{t\to-\infty} {v^*}(t)$.
Integrating \rf{32}
from $-R$ to $0$ (where $R>0$ is sufficiently large) and taking into account that $e^{f({v^*})} \gt e^{f(L)}$, we obtain that
\aa{
{v^*}'(-R)^{\beta+1} \gt e^{-(\al-\beta-1)R}{v^*}'(0)^{\beta+1}  +\frac{\beta+1}{\gm+1}e^{f(L)}\big(e^{\te R}-e^{-(\al-\beta-1)R}\big)
}
which shows that ${v^*}'(-R) \to +\infty$ as $R\to +\infty$. The latter implies that
${v^*}(-R) = {v^*}(0) - \int_{-R}^0 {v^*}'(t) dt \to -\infty$ as $R\to +\infty$. This contradicts to the fact that $\lim_{t\to-\infty} {v^*}(t)=L<+\infty$.

Thus, we conclude 
that solution \rf{sol-g1} can be extended to $(-\infty,+\infty)$ in such a way that there exists a finite number 
$T^*$ such that $v^*(T^*) = 0$.
\end{proof}


\subsubsection{Proof of Theorem \ref{thm1}}
The proof of Theorem \ref{thm1} now follows from the inverse change of  variable.
\begin{proof}[Proof of Theorem \ref{thm1}]
Let $T^*$ be as in Proposition \ref{pro99}.
Define $\la^* = (\beta+1)e^{-\te T^*}$ and $\kappa = \big(\frac{\beta+1}{\la^*}\big)^\frac1\te$. 
Further define  $u^*_{\lm^{\!*}}(r) = v^*(\ln(\kappa / r))$.
By the argument in paragraph \ref{s223}, $(u^*_{\lm^{\!*}},\la^*)$ is a singular solution to problem \rf{G}.
\end{proof}
  \subsection{Exact singular solution to the Gelfand problem for a $k$-Hessian}
  Consider the problem
  \eq{
  \lb{L}
- L(u) = \la \, e^u & \text{ in } B, \\
u = 0  & \text{ on } \pl B,
  }
  where $L$ can be a $k$-Hessian with $k<\frac{d}2$, a $p$-Laplacian with $p<d$, or, in general,
  the operator \rf{L-op} with $\al,\beta,\gm$ satisfying (A6). Here, we prove Corollary \ref{cor1111},
  stating the exact form of the singular solution constructed in Theorem \ref{thm1}. Furthermore, 
  we ``decode'' the formula for a singular solution announced in Corollary \ref{cor1111}
  for the cases of a $k$-Hessian and a $p$-Laplacian. Namely, we have the following corollary.
  \begin{cor}
  \lb{cor-hes}
  Assume (A6) and consider problem \rf{L} for the case when
 $L$ is a $k$-Hessian with $k<\frac{d}2$. Then,  the singular solution constructed in Theorem \ref{thm1} takes the form
  \aaa{
  \lb{sol-khes}
 u^*_{\lm^{\!*}}(x) = -2k\ln|x|, \quad \la^* = (2k)^k(d-2k).
  }
 If $L$ is a $p$-Laplacian with $p<d$, then, the aforementioned
singular solution is
\aaa{
  \lb{sol-plap}
 u^*_{\lm^{\!*}}(x) = -p\ln|x|, \quad \la^* = p^{p-1}(d-p).
  }
  \end{cor}
   \begin{proof}[Proof of Corollaries \ref{cor1111} and \ref{cor-hes}]
Note that for problem \rf{L},  $f(t) =g(t) = t$. Therefore, $g'(t) = 1$, $g''(t) = 0$.
By Proposition \ref{lem31}, in a neighborhood of $+\infty$, the solution  \rf{sol-g1} takes the form
\aaa{
\lb{vtu}
v^*(t) = \te t + \ln\{\te^{\beta+1}(\hat\al -1)\}. 
}
It is straightforward  to check that \rf{vtu} verifies \rf{eq-g1} with $f(v) = v$ on the entire real line.
Furthermore, $v(t)=0$ if $t=T^* = -\frac{\ln\{\te^{\beta+1}(\hat\al -1)\}}{\te}$. Therefore,
\aa{
\la^* = (\beta+1)e^{-\te T^*}  = (\beta+1) \te^{\beta+1}(\hat\al -1) = \te^{\beta+1}(\al - \beta -1).
}
Recall that $u^*_{\lm^{\!*}}(x) = v^*(t)$ with $t = \ln\big\{\big(\frac{\beta+1}{\la^*}\big)^\frac1\te \frac1{|x|}\big\}
 = - \frac1{\te}\ln \{ \te^{\beta+1} (\hat\al -1)\} - \ln|x|$.
Comparing it with \rf{vtu}, we obtain that $u^*_{\lm^{\!*}}(x) = -\te \ln|x|$. 

In the case of a $k$-Hessian with $k<\frac{d}2$, we have $\te = 2k$, $\beta+1 = k$, $\al - \beta -1 = d-2k$, which implies
\rf{sol-khes}.

In the case of a $p$-Laplacian with $p<d$, we have $\te = p$, $\beta+1 = p-1$, $\al - \beta -1 = d-p$, which implies
\rf{sol-plap}.
\end{proof}

\section{Regular solutions to problem \rf{G}}
\lb{s3}
 Here, we will be interested in regular solutions to problem \rf{G}. For regular solutions, problem \rf{G}
should be stated as follows:
 \eq{
 \lb{GG2}
  \tag{$P'_{\la}$}
 - L(u_\lm(r)) =  \la \, e^{f(u_\lm(r))}, \\
  u'_\lm(0) = u_\lm(1) = 0.
 }
 Indeed, from Lemma \ref{rem98} below, it follows that any solution to \rf{G} is decreasing  in $r\in [0,+\infty)$  and achieves the maximum at $r=0$.
 Moreover, since the actual solution $u_\lm$ should be regular in a ball, we add the condition  $u'_\lm(0) = 0$.
\begin{lem}
\lb{rem98}
Let $(u_\lm,\la)$ be a regular radial solution to \rf{GG2}. Then, $u'_\lm(r)<0$ on $(0,+\infty)$.
\end{lem}
\begin{proof}
 The equation in \rf{GG2} implies 
\aa{
 r^\al |u'_\lm(r)|^\beta u'_\lm(r) = -\int_0^r s^\gm e^{f(u_\lm(s))} ds + C.
}
Since $u'_\lm(0) = 0$, we obtain that $C=0$. 
This implies that  $u'(r)< 0$ for all $r> 0$.
\end{proof}
\begin{rem}
\lb{rem554}
\rm
If $\gm+1>\al$ (as in the case of the $k$-Hessian), then in \rf{GG2}, we do not need to state that $u'(0) = 0$ if we are interested in a regular solution. This assumption is fulfilled automatically. Indeed, by L'Hopital's rule,
\aa{
(-u'_\lm(0))^{\beta+1} = \lim_{r\to 0} \al^{-1} r^{\gm+1-\al} e^{f(u_\lm(r))} = 0.
}

\end{rem}
\subsection{Standing assumptions: regular solutions}
\lb{s3.1}
\bi
\item[(B1)] $f\in \C^1(\Rnu,\Rnu)$, $f'>0$, and $\lim_{u\to+\infty} f(u) = +\infty$.
\item[(B2)] $\lim_{u\to+\infty} u^{-1}\exp\big\{\frac{f(u)}{\beta+1}\big\} = \lim_{u\to+\infty} u f'(u)  = +\infty$.
\item[(B3)] $\al>\beta+1>0$; $\te>0$. 
\ei
\begin{rem}
\rm
Remark that the second limit in (B2) is implied by (A4) and  the first condition of (A2) ($\lim_{t\to+\infty} g''(t) = 0$). 
We have
\aaa{
 \lb{lim99}
 \lim_{u\to +\infty} u f'(u)
=  \lim_{t\to +\infty} \frac{g(t)}{g'(t)}
 \gt  \lim_{t\to +\infty} \frac{g(t)}{\int_0^t |g''(s)| ds + g'(0)}= +\infty.
}
Indeed, we know that $\lim_{t\to+\infty} g(t) = +\infty$. If $\lim_{t\to+\infty} \int_0^t |g''(s)| ds<+\infty$, then
 \rf{lim99} is straightforward. Otherwise, if $\lim_{t\to+\infty} \int_0^t |g''(s)| ds =+\infty$, then \rf{lim99} follows from  L'Hopital's rule
 and Remark \ref{rem16}. 
 The first expression in (B2)  tends to $+\infty$ by L'Hopital's rule and Remark \ref{A6}.

\end{rem}
\subsection{Existence and uniqueness of a regular solution to  \rf{GG2}}
\lb{sec3.1}
First of all, we note that by the change of variable  \rf{ch-var},
problem \rf{GG2} can be transformed to
 \eq{
 \lb{G2}
 \tag{$P_\rho$}
 - L(u) =  e^{f(u)}, \\
  u(0) = \rho, \quad u'(0) = 0.
 }
 Above, $\rho$ and $\la$ are connected through the identity
 \aa{
 \rho = \la^{\frac1{\beta+1}} \int_0^1 t^{-\frac{\al}{\beta+1}}\Big(\int_0^t s^\gm e^{f(u_\lm(s))} ds\Big)^{\frac1{\beta+1}} dt,
 }
where $(u_\lm,\la)$ is the solution to \rf{GG2}. Note that problem \rf{G2} is equivalent  to the integral equation
\aaa{
\lb{G3}
u(r) = \rho - \int_0^r t^{-\frac{\al}{\beta+1}}\Big(\int_0^t s^\gm e^{f(u(s))} ds\Big)^{\frac1{\beta+1}} dt.
}
We have the following result on the existence and uniqueness of solution to \rf{G3}.
\begin{lem}
\lb{lem311}
Assume (B1) and (B3). Then, equation \rf{G3} has a unique solution.
\end{lem}
\begin{proof}
Let the map $\Gm(u)$ be given by the right-hand side of \rf{G3}.  Fix an arbitrary interval $[0,T]$. 
Note that for any solution $u$ of \rf{G3}, on $[0,T]$, it holds that
\aa{
u\lt \rho \quad  \text{and} \quad 
u \gt \rho - e^{\frac{f(\rho)}{1+\beta}} (\gm+1)^{-\frac1{\beta+1}} \hat\te^{-1} T^{\hat \te}.
}
Let $A_{\rho,T}$ denote the constant on the right-hand side of the second estimate.
Introduce the complete metric space
\aa{
\Sg = \{u\in \C[0,T]: \;  A_{\rho,T} \lt u\lt \rho\}
}
with the supremum norm as the metric,
and take two functions $u_1$ and $u_2$ from $\Sg$. It is easy to see that 
$\Gm(u_1)$ and $\Gm(u_2)$ are functions $\Sg\to \Sg$.
 Furthermore, 
 \mm{
 \Gm(u_1)(r) - \Gm(u_2)(r) = \frac1{\beta+1}\int_0^r t^{-\frac{\al}{\beta+1}} \Big(\int_0^1
\La(t,\rho,\la)^{-\frac{\beta}{\beta+1}} d\la\Big)\\
 \int_0^t s^\gm(e^{f(u_2(s))} - e^{f(u_1(s))}) ds\,  dt,
 }
 where  $\La(t,\rho, \la) = \int_0^t s^\gm(\la e^{f(u_1(s))} +(1-\la) e^{f(u_2(s))}) ds$, and we have that
 \aa{
   e^{f(A_{\rho,T})}(\gm+1)^{-1} t^{\gm+1}\lt \La(t,\rho, \la) \lt  e^{f(\rho)}(\gm+1)^{-1} t^{\gm+1}.
 }
 Taking into account that $|u_1|$ and $|u_2|$ are bounded,
 we obtain that there exists a constant $K=K(\rho,T,\gm,\beta,f)$ such that
 \aa{
 \sup_{[0,r]}|\Gm(u_1)- \Gm(u_2)|\lt K \int_0^r t^{\frac{-\al+\gm+1}{\beta+1}} \sup_{[0,t]}|u_1 - u_2| \, ds.
 }
Since $\frac{-\al+\gm+1}{\beta+1} = \hat \te -1$,
 we obtain that for the map $\Gm^{(n)}= \underbrace{\Gm \circ \Gm \circ \dots \circ \Gm}_n$, it holds that
  \aa{
 \sup_{[0,r]}|\Gm^{(n)}(u_1)- \Gm^{(n)}(u_2)|\lt \frac{(K r^{\hat \te} \hat \te^{-1})^n}{n!}  \sup_{[0,r]}|u_1 - u_2| \, ds.
 }
 This implies that  $\Gm^{(n)}$ is a contraction for some $n$. Clearly, the unique fixed point of $\Gm^{(n)}$ is also the
 unique fixed point of $\Gm$, the solution to \rf{G3}.
\end{proof}
\subsection{Properties of regular solutions}
\lb{s3.3}
We start by deriving a version of Pohozaev's identity suitable for  our applications.
 \begin{lem}
 \lb{pohoz}
Let $f\in \C(\Rnu,\Rnu)$,  $a\in \Rnu$ be an arbitrary constant, and
$u \in \C^2[0,+\infty)$ be a solution to problem \rf{G2}.
 Then, for all $r>0$,
 \mmm{
 \lb{PI}
 r^\al  |u'|^{\beta+2}\Big(a-\frac{\al-\beta-1}{\beta+2}\Big) + r^\gm \Big\{(\gm+1)  \int_0^u e^{f(t)} dt - a \, u \, e^{f(u)}\Big\}\\
 = \frac{d}{dr} \Big\{ r^{\al+1} \Big(\frac{\beta+2}{\beta+1}\, |u'|^{\beta+2} + r^{\gm-\al}  \int_0^u e^{f(t)} dt + ar^{-1} |u'|^\beta u' u\Big)\Big\}.
 }
 \end{lem}
 \begin{proof}
 In \cite{ni89}, a version of Pohozaev's identity was obtained for solutions to equations whose particular type is
 \aaa{
 \lb{dd}
 r^{-(d-1)} (r^{(d-1)} |u'(r)|^\beta u'(r))' + h(r,u) = 0, 
 }
 where $d$ is the space dimension and $h\in\C(\Rnu_+\x \Rnu)$.  
As it follows from the proof  of Theorem 2.1,  exposed in paragraph 2 of \cite{ni89}, instead of $d-1$, one can use
any number $\al>0$. More specifically, the fact that $d-1$ is an integer which differs from the space dimension by one, 
does not play any role in the computation on p. 545 of \cite{ni89}.
Multiplying the both parts of  the equation in \rf{G2} by $r^{\al-\gm}$, we bring this equation to a form similar to \rf{dd}:
\aa{
r^{-\al}(r^\al |u'(r)|^\beta u'(r))'  + r^{\gm-\al} e^{f(u)} = 0.
}
We then apply Theorem 2.1 from \cite{ni89} to the above equation, substituting  $d$ by $\al+1$ and
taking into account that $u'(0) = 0$. Referring to the notation of the aforementioned theorem, we have $A(\rho) = \rho^\beta$ and $E(\rho) = (\beta+1) \rho^\beta$. It is then straightforward to verify the assumptions of  Theorem 2.1 from \cite{ni89}  as well as to make
necessary computations leading to equation \rf{PI}.
 \end{proof}
 Let $u(\fdot,\rho)$ be the solution to problem \rf{G2}.
 Our next result, given by   Proposition \ref{lrhoB}, is about upper and lower bounds for the number
\aaa{
\lb{rhoB}
R(B,\rho) = u(\fdot,\rho)^{-1}(B), \quad 0\lt B\lt \rho.
}
 \begin{pro}
 \lb{lrhoB}
 Assume (B1)--(B3).
Then, for each $B\gt 0$, the number $R(B,\rho)$ is well-defined.
 Moreover, there exist two positive constants $\underline{R}(B)$ and $\bar R(B)$, and a number $\rho_0(B)>B$,
 all depending only on $B$, 
 such that for all $\rho\gt \rho_0(B)$,
 \aa{
 \underline{R}(B) \lt R(B,\rho) \lt \bar R(B).
 }
 Moreover, the second inequality holds for all $\rho\gt B$.
 \end{pro}
\begin{proof}
Fix a number $B\gt 0$. 
We start by  showing that $R(B,\rho)$ is well-defined for $\rho> B$. This is equivalent to the fact that a zero of $u(r,\rho)-B$ exists. 
Suppose this is not the case. Since $u(r,\rho)$ is decreasing, then $u(r,\rho) > B$ for all $r\gt 0$.
From \rf{G2}, we obtain that
 $u(r,\rho) \lt 
 \rho - e^{\frac{f(B)}{1+\beta}} (\gm+1)^{-\frac1{\beta+1}} \hat\te^{-1} r^{\hat \te} \to -\infty$, as $r\to +\infty$.
 This is a contradiction. Hence, there exists a unique $r$ such that $u(r,\rho) = B$.

Let us prove that  the function $[B,+\infty)\to \Rnu$, $\rho\mto R(B,\rho)$ is bounded from above. 
Rescale problem \rf{G2}  by setting $\la^{\frac1\te} = R(B,\rho)$ as follows:
  \eqq{
 - L(u(\la^{\frac1\te}r)) =  \la \, e^{f(u(\la^{\frac1\te} r))}, \\
  u(0) = \rho, \quad u'(0) = 0
 }
 which is equivalent to
  \aa{
u(\la^{\frac1\te} r) =  B+ \int_r^1 \la^{\frac1{\beta+1}} t^{-\frac{\al}{\beta+1}}\Big(\int_0^t s^\gm e^{f(u(\la^{\frac1\te} s))} ds\Big)^{\frac1{\beta+1}} dt.
}
For $r\in (0,1)$, it holds that
\aa{
u(\la^{\frac1\te} r)- B \gt  \la^{\frac1{\beta+1}}   \int_r^1  t^{-\frac{\al}{\beta+1}} dt \, \Big(\int_0^r s^\gm e^{f(u(\la^{\frac1\te} r))} ds\Big)^{\frac1{\beta+1}}.
}
 This implies that for $r\in (0,1)$, we have the estimate
 \aaa{
 \lb{est667}
 \sup_{u\gt B} \big\{(u-B) \, e^{-\frac{f(u)}{\beta+1}}\big\} \gt (u(\la^{\frac1\te} r) - B) e^{-\frac{f(u(\la^{\frac1\te} r))}{\beta+1}}
 \gt K_1 \la^{\frac1{\beta+1}} (r^{\frac\te{\beta+1}} - r^{\frac{\gm+1}{\beta+1}}),
 }
 where $K_1>0$ is a constant.
By (B2), $0<\sup_{u\gt B} (u-B) \, e^{-\frac{f(u)}{\beta+1}}<+\infty$.
Hence, the right-hand side of \rf{est667} is bounded uniformly in $r\in (0,1)$. Evaluating it (say) at $r=\frac12$, we obtain that
for all $B\gt 0$ and $\rho\gt B$,
\aaa{
\lb{K2}
R(B,\rho) = \la^{\frac1\te} \lt  K_2 \sup_{u\gt B}\big\{ (u-B) \, e^{-\frac{f(u)}{\beta+1}}\big\}^{\frac{\beta+1}\te} = \bar R(B),
}
where   $K_2$ is a constant. 
 Now we prove the existence of a lower bound $\underline{R}(B)$ by using identity \rf{PI}.
By L'Hopital's rule and (B2),
 \aa{
 \lim_{u\to +\infty} \frac{u\, e^{f(u)}}{\int_0^u e^{f(t)} dt} =  \lim_{u\to +\infty} \frac{e^{f(u)} + u\, e^{f(u)} f'(u)}{e^{f(u)}}
 =  1 + \lim_{u\to +\infty}  uf'(u) = +\infty.
  }
 As we show below, it suffices to prove the existence of  $\underline{R}(B)$ for those $B>0$ which make the following
 inequality hold:
 \aaa{
 \lb{est18}
 (\gm+1)  \int_0^u e^{f(t)} dt <  a u\, e^{f(u)}  \quad \text{for all} \; \; u \gt B,
 }
 where $a\in\Rnu$ is to be fixed later.
 Indeed, suppose  the existence of  $\underline{R}(B)$  is obtained for those $B$ which make inequality \rf{est18} fulfilled for $u\gt B$.
If $B\gt 0$ is arbitrary, we find a number $\td B>B$ such that \rf{est18} holds for $u\gt \td B$.
Then, $\underline{R}(\td B) \lt R(\td B,\rho) \lt R(B,\rho)$ for $\rho\gt \rho_0(\td B)$, and we can define $\underline{R}(B) = \underline{R}(\td B)$ 
and $\rho_0(B) = \rho_0(\td B)$. 
Thus, without loss of generality, we assume that $B$ is
sufficiently large so that \rf{est18} holds.

The number $a$ is chosen as follows.  Note that for any two positive numbers $x$ and $y$, it holds that
$(x+y)^{\frac1{1+\beta}} \lt x^{\frac1{1+\beta}} + y^{\frac1{1+\beta}}$ if $\beta\gt 0$ and 
$(x+y)^{\frac1{1+\beta}} \lt 2^{-\frac{\beta}{\beta+1}} (x^{\frac1{1+\beta}} + y^{\frac1{1+\beta}})$ if 
$-1<\beta<0$. We set $A_\beta = \max\{1,2^{-\frac{\beta}{\beta+1}}\}$ and
 fix $\hat a\in (0, A_\beta^{-1}- \frac{a (\beta+1)^2}{(\beta+2)(\al-\beta-1)})$. The number
 $a>0$ is then has to be chosen in such a way that $\frac{a (\beta+1)^2}{(\beta+2)(\al-\beta-1)}<A_\beta^{-1}$.
In addition, we choose $a$ to satisfy $a<\frac{\al-\beta-1}{\beta+2}$ so that the first term on the left-hand side
of \rf{PI} is negative.
By \rf{PI} and \rf{est18}, on $(0,\bar R(B)]$, it holds that
 \aa{ 
 \frac{\beta+2}{\beta+1}\, |u'|^{\beta+2} + ar^{-1} |u'|^\beta u' u <0.
 }
Since $u'<0$, for $r\in (0,\bar R(B)]$,
 \aaa{
 \lb{estm568}
 r\, |u'(r)| < a \, \frac{\beta+1}{\beta+2} \, u(r).
 }
Next, we define 
\aa{
 D = B\Big(1-\hat a A_\beta  -\frac{ A_\beta\, a (\beta+1)^2}{(\beta+2)(\al-\beta-1)}\Big)^{-1} \; \text{and}
 \quad  \underline{R}(B)= (\hat a \, D \, \hat\te (\gm+1)^\frac1{\beta+1} e^{-\frac{f(D)}{\beta+1}} )^\frac1{\hat\te}.
 }
 We aim to prove that $R(B,\rho)\gt \underline{R}(B)$ for all $\rho\gt D$.
  If, for some $\rho\gt D$, $R(D,\rho)\gt \underline{R}(B)$, then (since $D>B$)
 $R(B,\rho)> \underline{R}(B)$. Otherwise, if $R(D,\rho)< \underline{R}(B)$,
 consider the  following problem on $[R(D,\rho),\underline{R}(B)]$:
\aaa{
\lb{probD}
L(v) = -e^{f(D)}, \quad v(R(D,\rho)) = D, \quad  v'(R(D,\rho)) = u'(R(D,\rho), \rho).
}
To simplify notation, define $\td D = e^{f(D)}$ and $\td R = R(D,\rho)$.
Since
\aaa{
\lb{eq60}
r^\al |v'|^\beta v' =  -\td D\int_{\td R}^r s^\gm ds - \td R^\al |u'(\td R)|^{\beta+1},
}
we obtain that on  $[\td R,\underline{R}(B)]$, $v'(r)<0$.  Taking into account this observation, 
 by \rf{estm568} and \rf{eq60}, for  $r\in[\td R,\underline{R}(B)]$, we obtain
\mm{
v(r) = D - \int_{\td R}^r \Big(\frac{\td D}{\gm+1} (s^{\gm+1} -\td R^{\gm+1}) + \td R^\al |u'(\td R)|^{\beta+1}\Big)^\frac1{\beta+1} s^{-\hat \al} ds\\
\gt D- A_\beta\Big[\Big(\frac{\td D}{\gm+1}\Big)^\frac1{\beta+1} \int_{\td R}^r s^{\frac{\gm+1}{\beta+1}} s^{-\frac{\al}{\beta+1}} ds
+\frac{\td R^{\hat\al} |u'(\td R)|r^{-\hat\al+1} }{\hat\al-1} - \frac{ \td R\, |u'(\td R)|}{\hat\al-1}\Big]\\
> D - A_\beta \Big(\frac{\td D}{\gm+1}\Big)^\frac1{\beta+1} \frac{\underline{R}(B)^{\hat \te}}{\hat \te} 
- \frac{A_\beta \, a(\beta+1)D}{(\beta+2)(\hat\al-1)} = B,
}
where the last identity follows from the definition of $\underline{R}(B)$ and $D$. 

It remains to show that on $[\td R,\underline{R}(B)]$, $u(r,\rho)\gt v(r)$.  Indeed, by \rf{probD}, for 
$r\in [\td R,\underline{R}(B)]$, we have
\aa{
&(r^\al |v'|^{\beta+1})' = r^\gm e^{f(D)} \gt r^\gm e^{f(u)} = (r^\al |u'|^{\beta+1})' \text{  and}\\
&r^\al |v'|^{\beta+1} \gt r^\al |u'|^{\beta+1}.
}
Since the derivatives $u'$ and $v'$ are negative, again by \rf{probD}, we obtain that on $[\td R,\underline{R}(B)]$, 
$u(r,\rho)\gt v(r) \gt B$. Therefore, $R(B,\rho)\gt \underline{R}(B)$, and we can take $\rho_0(B) = D$.
The proof is now complete.
 \end{proof}
 \begin{pro}
\lb{r-comp}
Assume (B1)--(B3). 
Then, for any interval $[r_1,r_2]\sub (0,\infty)$,
the family $\{u(\fdot,\rho)\}_{\rho\gt 0}$ is relatively compact in $\C([r_1,r_2])$.
Furthermore, each limit point of this family is a solution to \rf{sing1}, regular or singular.
 \end{pro}
\begin{proof}
Fix an interval $[r_1,r_2]\sub (0,+\infty)$. 
Let $B>0$ and let $\underline{R}(B)$, $\bar R(B)$, and $\rho_0(B)$ be as they were defined in    Proposition \ref{lrhoB}  ,
so we have that  $R(B,\rho) \in [\underline{R}(B),\bar R(B)]$ for all $\rho>\rho_0(B)$, where $R(B,\rho)$ is defined by \rf{rhoB}.
By (B2), $\lim_{B\to+\infty}  \bar R(B) = 0$.
Choose a number $B>0$ such that $r_1>\bar R(B)$, and note that on $[r_1,r_2]$, $u(r,\rho) < B$ for all $\rho\gt B$.
By \rf{G3}, it holds that for $r\gt R(B,\rho)$, 
\mmm{
\lb{arg1}
u(r,\rho) = B - \int_{R(B,\rho)}^r t^{-\frac{\al}{\beta+1}}\Big(\int_0^t s^\gm e^{f(u(s,\rho))} ds\Big)^{\frac1{\beta+1}} dt  \\
\gt B -  e^{\frac{f(B)}{\beta+1}}\int_{R(B,\rho)}^r  t^{-\frac{\al}{\beta+1}}\Big(\int_0^t s^\gm ds\Big)^{\frac1{\beta+1}} dt
\gt     B \, - \, e^{\frac{f(B)}{\beta+1}} (\gm+1)^{-\frac1{\beta+1}}\hat \te^{-1} r^{\hat\te}.
}
Therefore, the family $\{u(r,\rho)\}_{\rho \gt B}$ is uniformly bounded on $[r_1,r_2]$.

Let us show that the family $\{u(r,\rho)\}_{\rho \lt B}$ is uniformly bounded on $[0,r_2]$. 
Note that this family is bounded from above by $B$.
It is also bounded from below by the same argument as \rf{arg1}. Namely,
\aa{
u(r,\rho) =  \rho - \int_{0}^r t^{-\frac{\al}{\beta+1}}\Big(\int_0^t s^\gm e^{f(u(s,\rho))} ds\Big)^{\frac1{\beta+1}} dt  
\gt  - e^{\frac{f(B)}{\beta+1}} (\gm+1)^{-\frac1{\beta+1}}\hat \te^{-1}r_2^{\hat\te}.
}
Furthermore, \rf{G3} implies the uniform continuity of the family $\{u(r,\rho)\}_{\rho\gt 0}$ on $[r_1,+\infty)$. Indeed,
for all $r',r''$, such that $r_1<r'<r''$, and $\rho\gt 0$,
\mmm{
\lb{arg2}
0<u(r',\rho) - u(r'',\rho) = \int_{r'}^{r''} t^{-\frac{\al}{\beta+1}}\Big(\int_0^t s^\gm e^{f(u(s,\rho))} ds\Big)^{\frac1{\beta+1}} dt\\
\lt e^{\frac{f(B)}{\beta+1}} (\gm+1)^{-\frac1{\beta+1}}\hat \te^{-1} ({r''}^{\hat \te} - {r'}^{\hat \te}).
}
By the Arzel\`a-Ascolli theorem, the family $\{u(r,\rho)\}$ is relatively compact on $[r_1,r_2]$.

Let us prove now that each limit point of the family $\{u(r,\rho)\}_{\rho\gt 0}$ is either regular or singular solution.
Let $\{u(r,\rho_n)\}$ converge to a function $u^{**}(r)$ in $\C([r_1,r_2])$.
We claim that the sequence
$\pl_r u(r_1,\rho_n)$ is bounded. Suppose it is not the case, that is, we can find a subsequence 
$\pl_r u(r_1,\rho_{n_k})\to -\infty$ as $k\to\infty$. Without loss of generality, we assume that $\lim_{n\to\infty} \pl_r u(r_1,\rho_n) = -\infty$. 
For $r\in[r_1,r_2]$, it holds that,
\aa{
u(r,\rho_n) = u(r_1,\rho_n) -
  \int_{r_1}^r t^{-\frac{\al}{\beta+1}} \Big( r_1^\al |\pl_r u(r_1,\rho_n)|^{\beta+1}
  + \int_{r_1}^t s^\gm e^{f(u(s,\rho_n))} ds 
 \Big)^\frac1{\beta+1} dt.
}
The last term in this identity converges to $\infty$ which contradicts to the convergence of  $\{u(r,\rho_n)\}$  in $\C([r_1,r_2])$.
Therefore, $\pl_r u(r_1,\rho_n)$ is bounded. Let  a subsequence $r_1^\al|\pl_r u(r_1,\rho_{n_k})|^{\beta+1}$ converge to $k(r_1)$. 
In the previous expression, passing to the limit as $k\to \infty$, we obtain
\aa{
u^{**}(r) = u^{**}(r_1) -  \int_{r_1}^r t^{-\frac{\al}{\beta+1}} \Big(k(r_1)+ \int_{r_1}^t s^\gm e^{f(u^{**}(s))} ds
 \Big)^\frac1{\beta+1} ds.
}
From the above identity, we obtain that $u^{**}(r)$ is differentiable in $r$, $(u^{**})'(r)<0$, 
$k(r_1) = r_1^\al |(u^{**})'(r_1)|^{\beta+1}$, and $u^{**}(r)$ solves \rf{sing1} on $[r_1,r_2]$.
By extracting further subsequences from $\{u(r,\rho_n)\}$  and repeating the argument on compact subintervals
of $(0,+\infty)$, we can extend $u^{**}(r)$ to  a solution of \rf{sing1} on  $(0,+\infty)$. 
Note that $\{u(r,\rho_n)\}$  contains a subsequence that converges pointwise to 
$u^{**}(r)$ on $(0,+\infty)$.

Let us show now the following fact. Suppose $u(r,\rho_n)$ converges to $u^{**}(r)$
pointwise on $(0,+\infty)$. If $\rho_n$ is unbounded,
then  $u^{**}(r)$ is a singular solution, if  $\rho_n$ is bounded, then it is a regular solution. 
Let us show that, in the first case, 
 $\lim_{r\to+0} u^{**}(r) = +\infty$. Suppose this is not the case, so we can find
a sequence $r_n\to 0$ as $n\to +\infty$ such that  $\lim_{n\to \infty} u^{**}(r_n) = B <+\infty$. 
Since $(u^{**})'(r)<0$, 
$\sup_{r\gt r_n} u^{**}(r)  \lt u^{**}(r_n)$, which implies that $\sup_{r\gt 0} u^{**}(r) \lt B$.
Let $\rho_{n_k}\to +\infty$ be a subsequence of $\{\rho_n\}$.
By   Proposition \ref{lrhoB}, we can find
 $\underline{R}(2B)$ and $\rho_0(2B)$  such that whenever $\rho_{n_k}>\rho_0(2B)$,
 $u(r,\rho_{n_k}) > 2B$ on $(0,\underline{R}(2B))$. Therefore, $u^{**}(r)\gt 2B$ on $(0,\underline{R}(2B))$.
 This contradicts to the previously obtained inequality $\sup_{r\gt 0} u^{**}(r) \lt B$.
Hence, $\lim_{r\to+0} u^{**}(r) = +\infty$.

Suppose now $\rho_n$, introduced above, is bounded.  Then, we can extract a subsequence $\lim_{k\to+\infty}\rho_{n_k}=\bar \rho< +\infty$.
Let us prove that $u^{**}(r)$ is a regular solution with $u^{**}(0) = \bar \rho$, i.e., 
 $u^{**}(r)= u(r,\bar\rho)$.  It follows from \rf{G3}, that $u(r,\rho)$ is continuous in $\rho$ uniformly
 in $r$ varying over compact intervals. Indeed,
 \mm{
\sup_{[0,r]} |u(\fdot,\rho_1) - u(\fdot,\rho_2)| \lt |\rho_1-\rho_2| \\ + \frac1{\beta+1}\int_{0}^r s^{-\frac{\al}{\beta+1}}
\Big(\int_0^1\La(s,\la)^{-\frac{\beta}{\beta+1}} d\la\Big)
 \int_0^s t^\gm(e^{f(u(t,\rho_2))} - e^{f(u(t,\rho_1))}) dt\,  ds\\
\lt  |\rho_1-\rho_2| + K_{r,\rho_1,\rho_2} \int_0^r s^{\hat \te - 1} \sup_{[0,s]}|u(\fdot,\rho_1) - u(\fdot,\rho_2)| ds, 
 }
where $K_{r,\rho_1,\rho_2}>0$ is a constant that only depends on $r,\rho_1,\rho_2$. Furthermore, above,
  \aa{
& \La(s,\la) = \int_0^s t^\gm(\la  e^{f(u(t,\rho_1))} +(1-\la) e^{f(u(t,\rho_2))}) dt,  \\
& K^{(1)}_{r,\rho_1,\rho_2}\, s^{\gm+1} \lt \La(s,\la)  \lt K^{(2)}_{r,\rho_1,\rho_2}\, s^{\gm+1},
 }
 where   $K^{(1)}_{r,\rho_1,\rho_2}$ and  $K^{(2)}_{r,\rho_1,\rho_2}$ are positive constants.
Now by Gronwalls's inequality, one can find another constant $K^{(3)}_{r,\rho_1,\rho_2}$, such that
\aa{
\sup_{[0,r]} |u(\fdot,\rho_1) - u(\fdot,\rho_2)| \lt K^{(3)}_{r,\rho_1,\rho_2} |\rho_1-\rho_2|. 
}
 This implies that
 $\lim_{k\to+\infty} u(r,\rho_{n_k}) = u(r,\bar\rho)$ on compact intervals of $[0,+\infty)$, and hence, $u^{**}(r)= u(r,\bar\rho)$.
\end{proof}
From the proof of Proposition  \ref{r-comp}, we obtain the following corollary. 
\begin{cor}
\lb{cor-comp}
Under assumptions of Proposition \ref{r-comp}, any limit point of the family $\{u(\fdot,\rho)\}_{\rho\gt 0}$ 
that can be obtained as a limit of a subsequence with $\rho_n\to+\infty$, is a singular solution. Otherwise, it is a regular solution.
\end{cor}
The theorem below is the main result of this subsection.
\begin{thm}
 \lb{thm311}
 Assume (B1)--(B3). Then, the result of Theorem \ref{thm3131} holds true.
\end{thm}
 \begin{proof}
Since we know that there exists a unique solution $u(r,\rho)$ to problem \rf{G2}, 
there is also a unique solution $(u_\lm,\la)$ with $u_\lm(r) = u(\la^{\frac1\te} r,\rho)$ 
to problem \rf{GG2}.  For $u$ and $u_\lm$ related through this identity, it holds that
 the zero  $R(0,\rho)$ of $u$ is expressed via $\la$ as follows: 
 \aaa{
 \lb{Rla}
 R(0,\rho) = \la^\frac1\te.
 }
By Proposition \ref{lrhoB}, $R(0,\rho)$ has an upper bound. Therefore,
the set of $\la$ with the property that $(u_\lm,\la)$ is a solution to \rf{GG2} is bounded.
 Let $\la^\#$ be the exact upper bound for this set.
Let us show that
for all $\la<\la^\#$, there exists a regular solution to problem \rf{GG2}.
Indeed, $R(0,\rho)$ is the unique solution to
the equation  $u(r,\rho) = 0$  with respect to $r$ for each fixed $\rho>0$.
 By the implicit function theorem, the function  $R(0,\fdot):(0,+\infty)\to \Rnu$ is continuous.
 Moreover, $\lim_{\rho\to 0} R(0,\rho) = 0$ which is implied by the inequality
 \mm{
\rho = u(0,\rho) - u(R(0,\rho),\rho) =
\int_0^{R(0,\rho)} t^{-\frac{\al}{\beta+1}}\Big(\int_0^t s^\gm e^{f(u(s,\rho))} ds\Big)^{\frac1{\beta+1}} dt\\
\gt 
 e^{\frac{f(0)}{\beta+1}} (\gm+1)^{-\frac1{\beta+1}}\hat \te^{-1} R(0,\rho)^{\hat \te}.
}
As a continuous function of $\rho$,  $\la = R(0,\rho)^\te$
takes all the values on the interval $[0,\la^\#)$. 
Consider the case $\la = \la^\#$. We prove that there exists a solution to \rf{sing1}, singular or regular, 
such that $(\la^\#)^\frac1\te$ is its zero.
Find a sequence $\{\rho_n\}$ such that $\la_n = R(0,\rho_n)^\te$ converges to $\la^\#$.
Let $u^{**}$ be a limit point of $u(\fdot,\rho_n)$ on  an interval $[a,b]\sub (0,+\infty)$ such that  $(\la^\#)^\frac1\te\in (a,b]$, i.e.,
there exists a subsequence $u(\fdot,\rho_{n_k})\to u^{**}$, as $k\to \infty$, uniformly on $[a,b]$.
By Corollary \ref{cor-comp}, if $\rho_{n_k}$ is unbounded, then $u^{**}$ is a singular solution. Otherwise,
it is a regular solution. Without loss of generality, we write $\rho_n$ for $\rho_{n_k}$.
Let $R^*(0)=  (\la^\#)^\frac1\te$. We have
\mmm{
\lb{fo89}
|u^{**}(R^*(0))| = |u(R(0,\rho_n),\rho_n) - u^{**}(R^*(0))| \\ \lt |u(R(0,\rho_n),\rho_n) - u(R^*(0),\rho_n)|  + |u(R^*(0),\rho_n) - u^{**}(R^*(0))|.
}
By the uniform continuity of 
the family $u(\fdot,\rho)_{\rho\gt 0}$, the first term on the right-hand side of \rf{fo89} goes to zero. The second term goes to zero by the uniform
convergence  $u(\fdot,\rho_n)\to u^{**}$ on $[a,b]$. This implies that $u^{**}((\la^\#)^\frac1\te)=0$, which completes the proof.
\end{proof}

\section{Study of regular solutions by means of the singular solution}
\lb{s4}

Here we demonstrate how the asymptotic representation \rf{as1} can be used to study regular solutions to problem \rf{G}.

The change of variable  \rf{ch-var} transforms problem \rf{G} to the following:
 \eqq{
- L(u) =  e^{f(u)}, & 0\lt r\lt \la^{\frac1\te}, \\
u(\la^{\frac1\te}) = 0.  
}
In the above problem, one can actually exclude $\la$ out of consideration keeping in mind that
$\la^{\frac1{\te}}$ is the zero of the solution (by   Proposition \ref{lrhoB}, we know that a zero of the solution exists).
In Subsections \ref{sub4.1} and \ref{sub4.2}, we deal with the singular solution 
${u}^*(r) = u^*_{\lm^{\!*}}((\la^*)^{-\frac1{\te}} r)$ 
and regular solutions 
$u(r,\rho) = u_\lm(\la^{-\frac1{\te}}r,\rho)$ to equation \rf{sing1}.  In the case of regular solutions, equation \rf{sing1}
is complemented with the initial conditions $u(0,\rho) = \rho$, $\pl_r u(0,\rho) = 0$.
\subsection{Transformation of problem \rf{G2} and its limit equation}
\lb{sub4.1}
 For simplicity of notation, in this section, we define the function
\aa{
\psi = e^f\!.
}
Note that if $\lim_{t\to+\infty} g''(t) = 0$, then, by Remark \ref{A6}, 
$\mc F(u)$, involved in \rf{transf}, is finite for all $u$. Indeed, by the change of
 variable $s = g(t)$, we obtain 
 \aaa{
 \lb{1111-}
 \int_u^{+\infty} e^{-\frac{f(s)}{\beta+1}} ds = \int_{f(u)}^{+\infty} e^{-\frac{t}{\beta+1}} g'(t) dt <+\infty.
 }
 \begin{rem}
 \rm 
 Sometimes we write
$\rho$ as the second parameter of $\td u$. However, we do not mean that $\rho$ is its maximal value, unlike
 the notation $u(r,\rho)$.
Note that 
 \aa{
& \td u(s, \rho) = \mc F^{-1}_1 \Big\{\mc F_1(1) \, \frac{\mc F(u(r,\rho))}{\mc F(\rho)}\Big\} \lt 
\mc F^{-1}_1(\mc F_1(1)) = 1, \quad s\gt 0, \\
 & \text{and} \quad \td u(0,\rho) = 1.
 }
 We just want to emphasize that $\td u$ actually depends on $\rho$. Moreover, we will study the limit of $\td u$
 as $\rho$ goes to infinity. 
 \end{rem}
 \begin{rem}
 \rm Sometimes we skip the dependence on $\rho$ in both $u$ and $\td u$ to simplify notation.
 \end{rem}
 \begin{lem}
 \lb{lem312}
 Assume \rf{1111-} is fulfilled.
 Let $u(r)$ be a $\C^2$-solution to problem \rf{G2}.
 Then, the function $\td u(s)$, defined by \rf{transf}, is a solution to the problem
 \eqn{
 \lb{eq-tr}
&\, L(\td u) + e^{\td u} + s^{\al-\gm} |\td u'|^{\beta+2}\Big(\frac{I(u(\eps_\rho s))}{\beta+1} - 1\Big) = 0,\\
&\, \td u(0) = 1, \quad \td u'(0) = 0,
 }
 where $\td u$ is everywhere evaluated at $s$ and $I(u) = \mc F(u) f'(u) \psi(u)^{\frac1{\beta+1}}$.
 \end{lem}
 \begin{proof}
 First, we note that the $\td u(s)$, defined by \rf{transf}, satisfies the initial conditions in \rf{eq-tr}.
 Define $G(u) = \mc F^{-1}_1 \eps_\rho^{-\hat \te} \mc F(u)$, where
 $\hat \te = \frac{\te}{\beta+1}$.
 Then, $u(r) = G^{-1}(\td u(s))$, where $r= \eps_\rho s$. 
 Computing $L(u)$ by using the  previous expression, we obtain
\mm{
L(u(r)) =  \eps_\rho^{-\te} \,s^{-\gm}\big(s^\al |\td u'(s)|^\beta \td u'(s) G'(u(r))^{-\beta-1}\big)'_s\\
= \eps_\rho^{-\te} \, L(\td u(s)) G'(u(r))^{-\beta-1} -  \eps_\rho^{-\te} \, (\beta+1) s^{\al-\gm} |\td u'(s)|^{\beta+2} G'(u(r))^{-\beta-3} G''(u(r)).
}
Problem \rf{G2} is then transformed to the following:
\aa{
L(\td u(s)) + \eps_\rho^{\te} \, \psi (u(r)) G'(u(r))^{\beta+1} - (\beta+1)  s^{\al-\gm} |\td u'(s)|^{\beta+2} \frac{G''(u(r))}{G'(u(r))^2} = 0.
}
By the definition of $\mc F(u)$ and $\mc F_1(u)$ in \rf{transf}, 
\aaa{
\lb{G'}
G'(u(r)) = \eps_\rho^{-\hat \te} e^{\frac{\td u(s)}{\beta+1}}\psi(u(r))^{-\frac1{\beta+1}},
}
which transforms  the second term into $e^{\td u(s)}$. Furthermore, 
\aaa{
\lb{f89}
\frac{G''(u)}{G'(u)^2} = \frac{(\ln G'(u))'}{G'(u)} =   \frac1{\beta+1}\frac{G'(u) - f'(u)}{G'(u)}. 
}
In \rf{f89},  $u$ is everywhere evaluated at $r$.
Recall that $\td u(s) = G(u(r))$. For simplicity of notation, we skip the dependence on $s$ in $\td u$ and, with a slight abuse of notation, 
simply write $\td u = G(u)$. By \rf{G'}, 
\aa{
\frac{f'(u)}{G'(u)}  = \frac{\psi'(u) \psi(u)^{-\frac\beta{\beta+1}} \mc F(u)}
{\eps_\rho^{-\hat \te}  \mc F(u) \, e^{\frac{\td u}{\beta+1}}} = \frac{I(u)}{\beta+1}
}
since $\eps_\rho^{-\hat \te}  \mc F(u) = \mc F_1(\td u)$ and $\mc F_1(\td u) = (\beta+1)e^{-\frac{\td u}{\beta+1}}$. 
The above computations  imply \rf{eq-tr}.
 \end{proof}
 \begin{lem}
 \lb{lem334}
 Assume (B1)--(B3). Then,
$\lim_{\rho\to+\infty} u(\eps_\rho s,\rho)= +\infty$ in $\C_{loc}[0,+\infty)$.
 \end{lem}
  \begin{proof}
 Fix an arbitrary interval $[0,T]$. 
 From    Proposition \ref{lrhoB}  it follows that
 for any $B > 0$,  there exists a  number $\underline{R}(B)>0$  and a constant $\rho_0>0$ such that for all $\rho>\rho_0$ and $r\in (0,\underline{R}(B))$, $u(r,\rho)>B$. Note that 
 $\lim_{\rho\to+\infty} \eps_\rho = 0$. This proves that
 $\lim_{\rho\to+\infty} \sup_{s\in[0,T]} u(\eps_\rho s,\rho) = +\infty$.
 \end{proof}
 \begin{lem}
 \lb{lem333}
 Assume \rf{1111-}, (B1)--(B3), and that $\lim_{t\to+\infty}\frac{g''(t)}{g'(t)} = 0$. Then,
 $\lim_{\rho\to\infty} I(u(\eps_\rho s)) = \beta+1$ in $\C_{loc}[0,+\infty)$.
 \end{lem}
  \begin{proof}
We have
 \aa{
 I(u) =  \frac{\int_u^{+\infty} e^{-\frac{f(s)}{\beta+1} } ds \, f'(u)}{e^{-\frac{f(u)}{\beta+1}}} = 
 \frac{\int_{f(u)}^{+\infty} e^{-\frac{t}{\beta+1} } g'(t) \, dt \, }{e^{-\frac{f(u)}{\beta+1}} g'(f(u))}.
 }
 Let $v=f(u)$. By L'Hopital's rule,
 \aa{
 \lim_{v\to+\infty}  \frac{\int_v^{+\infty} e^{-\frac{t}{\beta+1} } g'(t) \, dt }{e^{-\frac{v}{\beta+1}} g'(v)}
 = & \lim_{v\to+\infty} \frac{ e^{-\frac{v}{\beta+1} } g'(v)}{(\beta+1)^{-1} e^{-\frac{v}{\beta+1}} g'(v) - e^{-\frac{v}{\beta+1}} g''(v)}\\
= & \lim_{v\to+\infty}  \frac{ (\beta +1) }{1-(\beta+1)\frac{g''(v)}{g'(v)}} = \beta+1.
 }
 The statement now follows from Lemma \ref{lem334}.
 \end{proof}

 \begin{lem}
 \lb{lem35}
 Assume (A1)--(A6). Then,
 the following representation holds for the
 singular solution \rf{sol-g1} of equation \rf{eq-g1}{\rm :}
 \aaa{
 \lb{s-sol1}
 v^*(t) = \mc F^{-1}(A^{-1} \,e^{-\hat \te t} (1+\eps(t))), 
 }
 where $A = \hat\te (\hat \al -1)^{\frac1{\beta+1}}$ and \, $\lim_{t\to+\infty}\eps(t) = 0$.
 \end{lem}
 \begin{proof}
Identity \rf{s-sol1} is equivalent to  $\eps(t) = A\, e^{\hat\te t} \mc F(v^*(t)) - 1$. Let us show that
$\lim_{t\to+\infty} e^{\hat\te t} \mc F(v^*(t)) = A^{-1}$.
 Recall that $v^*(t) = g(\te t + \ffi(t)) + \eta(t)$, where $\ffi=\ffi_1+\ffi_2$ is defined by \rf{ffi}, $\eta = O(g''(\te t))$,  and $\eta' = O(g''(\te t))$.
 By L'Hopital's rule, 
 \aa{
 \lim_{t\to+\infty} \frac{\mc F(v^*(t))}{e^{-\hat \te t}} = \lim_{t\to+\infty} \frac{e^{-\frac{f(v^*(t))}{\beta+1}} (g'(\te t + \ffi) (\te +\ffi') +\eta')}
 {\hat\te e^{-\hat \te t}}.
 }
 Transforming $f(v^*(t))$ by formula \rf{eq5}, we obtain that the limit on the right-hand side equals to
 \aa{
 \lim_{t\to+\infty}  \hat \te^{-1} e^{-\frac{\ffi}{\beta+1}}g'(\te t+\ffi)(\te + \ffi')
 = \lim_{t\to+\infty}   \frac{g'(\te t+\ffi)(\te + \ffi')}
  {\hat\te \te (\hat\al -1)^{\frac1{\beta+1}}g'(\te t)} = A^{-1}.
 }
We used representation \rf{ffi} for $\ffi=\ffi_1+\ffi_2$ and  the fact that $\lim_{t\to+\infty} \frac{\eta}{g'(\te t  +o(t))} = 0$.
 Finally, $\ffi'$ goes to zero by \rf{tet1}. This concludes the proof.
 \end{proof}
 \begin{lem}
 \lb{lem13}
 Assume (A1)--(A6). Then,
 the singular solution $u^*$ to  \rf{sing1} has the representation
 \aa{
 u^*(r) = \mc F^{-1}\big(K\, r^{\hat \te}(1+\dl(r))\big),
 }
 where
 $K = \frac{\beta+1}{\te}(\al-\beta-1)^{-\frac1{\beta+1}}$, $\dl(r) = \eps(\ln {\kappa} -\ln {r})$, $\kappa =  (\beta+1)^\frac1\te$.
 \end{lem}
 \begin{proof}
 The proof follows immediately from Lemma \ref{lem35} by the substitution $t = \ln \kappa - \ln r$.
  The constant $K$ can be computed by the definition of $A$  as follows: $K = A^{-1} (\beta+1)^{-\frac1{\beta+1}}
  =\frac{\beta+1}{\te}(\al-\beta-1)^{-\frac1{\beta+1}}$.
 \end{proof}
  Lemma \ref{lem13} allows to obtain a representation for the singular solution $\td u^*$
  to \rf{eq-tr}.
 \begin{lem}
 \lb{lem37}
 Assume (A1)--(A6).  Then,
 the singular  to equation \rf{eq-tr}  $\td u^*(s) = \mc F^{-1}_1\!\eps_\rho^{-\hat \te} \mc F(u^*(\eps_\rho s))$ 
can be represented as follows:
 \aaa{
 \lb{tdu37}
 \td u^*(s) = \ln \{ \te^{\beta+1}(\al-\beta-1)\} -\te \ln s - (\beta+1)\ln \{1+\dl(\eps_\rho s)\}.
 }
\end{lem}
  \begin{proof}
 Formula \rf{tdu37} is the result of application of transformation \rf{transf} to $u^*(\eps_\rho s)$.
 Indeed, $\td u^*(s) = \mc F^{-1}_1(K\, s^{\hat \te}(1+\dl(\eps_\rho s)))$, where $\mc F^{-1}_1 (t) = -(\beta+1) \ln\frac{t}{\beta+1}$.
This immediately implies \rf{tdu37}.
 \end{proof}
 \begin{pro}
 \lb{lem98}
 Assume (A1)--(A6).  Then,
 for the singular and regular solutions to \rf{eq-tr}, it holds that
 \aaa{
 \lb{conv67}
 \lim_{\rho\to+\infty} \td u(s) = w(s) \quad \text{and} \quad \lim_{\rho\to+\infty} \td u^*(s) = w^*(s)
 }
 in  $\C_{loc}[0,+\infty)$ and  $\C_{loc}(0,+\infty)$, respectively. 
Here, $w(s)$ is the regular solution to the problem
 \aaa{
 \lb{e-eq}
 L(w)+ e^w = 0, \quad w(0) = 1, \quad w'(0) = 0,
 }
and  $w^*(s)$ is the singular solution to the equation in \rf{e-eq} given explicitly as follows:
 \aaa{
 \lb{form-sing}
w^*(s) = \ln\{\te^{\beta+1}(\al-\beta-1)\} - \te\ln s = \ln \{\te^{\beta+1}(\hat\al-1)\} + \te\ln\frac{\kappa}{s}.
 }
  \end{pro} 
 \begin{rem}
 \rm
 The last expression in \rf{form-sing} agrees with  \rf{vtu} with $t =\ln \frac{\kappa}{s}$.
  \end{rem}
  For the proof of Proposition \ref{lem98}, we need the following lemma.
   \begin{lem}
Under assumptions of Lemma \ref{lem333},
 problem \rf{e-eq} is the limit problem for \rf{eq-tr} as $\rho\to +\infty$. More specifically,
   the last term on the right-hand side of \rf{eq-tr} goes to zero in $\C_{loc}[0,+\infty)$ as $\rho\to +\infty$.
  \end{lem}
  \begin{proof}
  Let $\td \om(s) = |\td u'(s)|^{\beta+1}$. Recall that by Lemma \ref{rem98}, $u'(r)<0$. Therefore, by \rf{transf}, 
  $|\td u'(s)| = - \td u'(s)$. Fix an arbitrary interval $[0,T]$ and introduce the notation,
  \aaa{
  \lb{irho}
   \mc I_\rho(s) = \frac{I(u(\eps_\rho s))}{\beta+1} - 1.
  }
Then, $\td \om(s)$ satisfies the following first-order differential  inequality:
  \aa{
  \td \om'(s) = s^{\gm-\al} e^{\td u} - \al s^{-1} \td \om + \td \om^{\frac{\beta+2}{\beta+1}} \mc I_\rho(s)
  \lt C + \eps \td \om + \eps \td \om^2.
  }
   Here, $C = T^{\gm-\al} e$ is the bound for $s^{\gm-\al} e^{\td u}$. Furthermore, $\eps$ is an arbitrary sufficiently small number.
Note that $|\mc I_\rho(s)| < \eps$ for sufficiently large $\rho$. Finally, it holds that $ \td \om^{\frac{\beta+2}{\beta+1}}  \lt \td \om + \td \om^2$.

We aim to estimate $\td \om(s)$  via the solution 
$\om(s)$ to the problem
\aaa{
\lb{om}
\om' = C + \eps \om + \eps \om^2, \qquad \om(0) = 0.
}
It is known that $\om = -\frac1{\eps} \frac{y'}{y}$ is a solution to the ODE in \rf{om} if $y$ is the solution to
\aa{
y'' - \eps y' + C\eps y = 0.
}
For sufficiently small $\eps$, the both roots of the characteristic equation are complex, and therefore, the solution takes the form
\aa{
y(s) = \bar Ke^{\frac{\eps}2 s} ( \cos (Ds) + K \sin (Ds)), \quad \text{where} \quad D = \frac{\sqrt{4C\eps - \eps^2}}2.
}
It is straightforward  to obtain now the solution  $\om$ to  problem \rf{om},
computing $K$ from the condition $\om(0) = 0$:
\aa{
\om(s) = \frac{\eps^2 + 4D^2}{4D\eps} \frac1{\ctg(Ds) - \frac{\eps}{2D}}.
}
Note that the solution $\om(s)$ is well-defined on the interval $\Big[0, \frac{2\arcctg\frac{\eps}{\sqrt{4C\eps - \eps^2}}}{\sqrt{4C\eps - \eps^2}}\Big)$.
When $\eps$ is sufficiently small, the numerator of the right endpoint is close to $\pi$; however, the denominator is small.
Thus, we can choose $\eps$ sufficiently small  (which makes $\rho$ sufficiently large)
such that $\om(s)$ is well-defined on $[0,T]$. Therefore, on  $[0,T]$, it holds that
\aaa{
\lb{est68}
 |\td u'(s,\rho)|^{\beta+1} \lt \om(s)
}
for sufficiently large $\rho$.
This, along with Lemma \ref{lem333}, proves the lemma.
  \end{proof}

 \begin{proof}[Proof of Proposition \ref{lem98}]
 Everywhere in the proof, $K_i$, $i=1,2,3$, are positive constants.
 Fix an arbitrary interval $[0,T]$. We have
 \mm{
 \hspace{-2mm} w(r) - \td u(r) = \frac1{\beta+1}\int_0^r t^{-\frac{\al}{\beta+1}} \Big(\int_0^1
\La(t,\rho,\la)^{-\frac{\beta}{\beta+1}} d\la\Big)\\
 \int_0^t (s^\gm(e^{\td u} - e^w) + s^\al |\td u'(s)|^{\beta+2} \mc I_\rho(s)) ds\,  dt,
 }
 where  $\La(t,\rho, \la)$ is defined by the first identity below and can be estimated by the expression on the right-hand side:
 \mm{
  \La(t,\rho, \la) = \la \int_0^t \big(s^\gm e^{\td u}+ s^\al |\td u'(s)|^{\beta+2}\mc I_\rho\big) ds
  +(1-\la) \int_0^t  s^\gm  e^w ds \gt (1-\la) e^{w(T)} \int_0^t  s^\gm  ds.
 }
Indeed, it holds that
 \aa{
 \int_0^t (s^\gm e^{\td u}  + s^\al |\td u'(s)|^{\beta+2} \mc I_\rho(s) ) ds =  t^\al |\td u'(t)|^{\beta+1} \gt 0.
 }
By Gronwall's inequality,
 \aaa{
  \lb{est8888}
 t^\al  |\td u'(t)|^{\beta+1} \lt e^{\int_0^t |\td u'(s) \mc I_\rho(s)| ds}  \int_0^t s^\gm e^{\td u} ds  \lt K_1 t^{\gm+1},
 }
 where the last estimate holds by \rf{est68} for sufficiently  large $\rho$.
Furthermore, estimate \rf{est8888} implies 
 \aa{
 \sup_{[0,r]} |\td u - w| \lt K_2 \int_0^r t^\frac{\gm+1-\al}{\beta+1} \big\{\sup_{[0,t]}|\td u - w|  + \sup_{[0,r]}|\mc I_\rho|\big\} \, dt,
 }
 where, in particular, it is taken into account that $\int_0^1(1-\la)^{-\frac\beta{\beta+1}} d\la = \beta+1$.
 By Gronwall's inequality, $\sup_{[0,r]} |\td u - w|  \lt K_3  \sup_{[0,r]}|\mc I_\rho| \, r^{\hat \te} e^{K_3 r^{\hat \te}}$.
This implies that $ \sup_{[0,T]} |\td u - w| \to 0$ as $\rho\to+\infty$.
 
We prove now the second convergence in \rf{conv67}.
It follows from the proof of Corollary \ref{cor1111} that the singular solution to \rf{e-eq} takes the form \rf{form-sing}.
If, in \rf{tdu37}, we pass to the limit as $\rho$ goes to $+\infty$, we obtain exactly \rf{form-sing}.
The convergence holds in $\C_{loc}(0,+\infty)$ since the convergence $\lim_{\rho\to+\infty}\dl(\eps_\rho s)=0$ holds
in this space.
  \end{proof}

 \subsection{Number of intersection points. Case $\al-\beta-1< \frac{4\te}{\beta+1}$ ($d<2k+8$).}
 \lb{sub4.2}
 In this subsection, we first study the number of intersection points between the regular and singular solutions $w$ and $w^*$ to  the equation
 in \rf{e-eq}, and then use this result to study the limit, as $\rho\to+\infty$, of the number
 of intersection points between $u(\fdot,\rho)$ and $u^*$.
 Recall that the number of zeros of $w-w^*$ on an interval $(R,S)$ is denoted by $\ms L_{(R,S)}(w-w^*)$.
  Although Theorem \ref{zeros1} has its own interest, we prove it in this subsection
 as an auxiliary result. 
 \begin{thm}
 \lb{zeros1}
 Assume (A6).
 Suppose $\al-\beta-1< 4\hat\te$. Then, for any $R>0$,
 $w$ and $w^*$  have an infinite number of intersection points on $(R,+\infty)$,
  i.e., $\ms L_{(R,\infty)}(w-w^*) = +\infty$.
  Furthermore, $\ms L_{(0,R]}(w-w^*)$ is finite.
 \end{thm}
 \begin{proof}
 Define $v(t) = w(s)$ and $\om(t) = v'(t)^{\beta+1}$, where $t=\ln\{\frac{\kappa}{s}\}$ with 
 $\kappa =  (\beta+1)^\frac1\te$.
 By Lemma \ref{rem98}, $v'(t)>0$, and hence, by \rf{32},
  \aaa{
  \lb{eq987}
  \big(e^{(\beta+1-\al)t} \om(t) \big)' = -(\beta+1) e^{-(\gm+1) t} e^{v}.
  }
  Now let $w^*(s)$ be given by \rf{form-sing} and  $v^*(t) = w^*(s)$, where  $t=\ln\{\frac{\kappa}{s}\}$.
Then,  $v^*(t) =  \te t + \ln\{\te^{\beta+1}(\hat\al-1)\}$.
Define $x(t) = v(t) - v^*(t)$ and $y(t) = \om(t) - \om^*(t)$, where $\om^*(t) = {v^*}'(t)^{\beta+1}= \te^{\beta+1}$.
Equation \rf{eq987} implies that $x$ and $y$ satisfy the following system of ODEs
\eq{
\lb{sys1}
y' = (\al -\beta-1)(y-\te^{\beta+1}(e^x-1)),\\
x' = (y+\te^{\beta+1})^\frac1{\beta+1} - \te,
}
where the last equation holds due to the fact that $v' = \om^{\frac1{\beta+1}}$.
Note that system \rf{sys1} has a unique steady state $(0,0)$. We are going to determine its type. 
The linearization of \rf{sys1} has the form
\eq{
\lb{linearized}
y' = (\al-\beta-1)(y - \te^{\beta+1}x),\\
x' = \frac{y}{(\beta+1)\te^\beta}.
}
Since $0<\al-\beta-1< 4\hat\te$, the roots of the characteristic equation for the linearized system
are complex numbers with the positive real part $\frac{\al-\beta-1}2$, and hence, $(0,0)$ is an unstable focus. 
This means that as $t$ goes to $-\infty$, $x(t)$ crosses zero infinitely many times at any neighborhood of $-\infty$.
This implies that for any $R>0$, $\ms L_{(-\infty,-R)}(v-v^*) = \ms L_{(\kappa e^R,+\infty)}(w-w^*) = +\infty$.

Suppose $\ms L_{(0,R]}(w-w^*)=+\infty$ for some $R>0$. Then, there exists an accumulation point, say $s_0$, of zeros of 
$w-w^*$ belonging to $[0,R]$. Note that $(w-w^*)(s_0) = 0$ which implies that $s_0\in (0,R]$. Furthermore,
$(w-w^*)'(s_0) = 0$. This implies that $v(t_0) = v^*(t_0)$ and $\om(t_0) = \om^*(t_0)$ for $t_0 = \ln \kappa - \ln s_0$,
which is a contradiction.
\end{proof}
\begin{cor}
\lb{cor55}
 Under assumptions of Theorem \ref{zeros1}, there exists $R>0$ such that $w-w^*$ changes the sign at each zero point
 on $(R,+\infty)$.
\end{cor}
\begin{proof}
The proof follows immediately from the fact that $(0,0)$ is a focus for \rf{sys1}.
\end{proof}
\begin{thm}
\lb{thm888}
Assume (A1)--(A6) and let
 $\al-\beta-1< 4\hat\te$.  Let $u(\fdot,\rho)$ 
 be the regular solution to problem \rf{G2} and 
 $u^*$ be the singular solution to \rf{sing1} defined by \rf{ch-var} 
 via $(u^*_{\lm^{\!*}},\la^*)$.
  Then, for any $\dl>0$, it holds that
 $\lim_{\rho\to +\infty} \ms L_{(0,\dl)}(u(\fdot,\rho)- u^*(\fdot)) = +\infty$. 
\end{thm}
\begin{proof}
 Recall that
 $\td u-\td u^*\to w- w^*$  as $\rho\to +\infty$ uniformly on any interval $[R,S]\sub (0,+\infty)$.
Note that for a sufficiently large $\rho$, $\ms L_{[R,S]}(\td u-\td u^*)$ is not smaller than
$\ms L_{[R,S]}(w- w^*)-2$, where the latter number is finite by Theorem \ref{zeros1}.
Recall that
$u(\eps_\rho s, \rho) = \mc F^{-1}_1 \eps_\rho^{\hat \te} \mc F(\td u(s,\rho))$ with $\lim_{\rho\to+\infty} \eps_\rho = 0$.
Fix $N\in\Nnu$. There exist $R,S>0$ and a number $\rho_{R,S}>0$ such that 
$\ms L_{[R,S]}(\td u-\td u^*) \gt N$ for $\rho>\rho_{R,S}$. Note that for each zero point $s\in [R,S]$ of $\td u-\td u^*$,
there exists exactly one zero point $r\in [\eps_\rho R, \eps_\rho S]$ of $u(\fdot,\rho)- u^*(\fdot)$.
Choose $\rho_\dl>\rho_{R,S}$ such that for all $\rho>\rho_\dl$, $[\eps_\rho R, \eps_\rho S]\sub (0,\dl)$. 
Therefore, for any $N\in \Nnu$, one can find $\rho_\dl$ such that for all $\rho\gt \rho_\dl$,
 $\ms L_{(0,\dl)}(u(\fdot,\rho)- u^*(\fdot))\gt N$. The theorem is proved.
\end{proof}
\subsection{Bifurcation diagram. Case $\al-\beta-1< \frac{4\te}{\beta+1}$ ($d<2k+8$).}
\begin{thm}
\lb{pro189}
Assume (A1)--(A6) and $\al-\beta-1< 4\hat\te$.
 Let $(u_{\sss\lm(\rho)}, \la(\rho))$  be  the family of regular solutions to problem \rf{G}  and let  
 $(u^*_{\lm^{\!*}}, \la^*)$ be the singular solution constructed in Theorem \ref{thm1}.
 Then, as $\rho\to\infty$, $\la(\rho)$ oscillates around $\la^*$. In particular, there are 
 infinitely many regular solutions to problem \rf{G} for $\la(\rho) = \la^*$. 
\end{thm}
   \begin{proof}
   Let $u(\fdot,\rho)$ and $u^*$ be the regular and singular solutions to \rf{sing1}, as in Theorem \ref{thm888}.
  Define $\hat v(r,\rho) =  u(r,\rho) - u^*$ and show that
 \aaa{
 \lb{prop1}
 \pl_r \hat v(r_0,\rho) \ne 0 \; \text{  whenever  } \; \hat  v(r_0,\rho)=0.
 }
Suppose $\hat  v(r_0,\rho) = \pl_r \hat v(r_0,\rho) = 0$. 
For $r\in (0,r_0)$, we have 
\mm{
u(r) - u^*(r) = \int_r^{r_0} s^{-\frac{\al}{\beta+1}}\Big[\Big(r_0^\al|u'(r_0)|^{\beta+1} 
- \int_s^{r_0} \xi^\gm e^{f(u(\xi))} d\xi\Big)^\frac1{\beta+1}\\
-\Big(r_0^\al|u'(r_0)|^{\beta+1} - \int_s^{r_0} \xi^\gm e^{f(u^*(\xi))} d\xi\Big)^\frac1{\beta+1}\Big]ds\\
 =\frac1{\beta+1}\int_r^{r_0} s^{-\frac{\al}{\beta+1}} \Big(\int_0^1
\La(s,\la)^{-\frac{\beta}{\beta+1}} d\la\Big)
 \int_s^{r_0} \xi^\gm(e^{f(u^*(\xi))} - e^{f(u(\xi))}) d\xi\,  ds,
 }
 where  
 \aa{
 \La(s,\la) = &\, r_0^\al|u'(r_0)|^{\beta+1} - \int_s^{r_0}\xi^\gm(\la e^{f(u(\xi))} +(1-\la) e^{f(u^*(\xi))}) ds\\
 = & \, s^\al(\la|u'(s)|^{\beta+1} + (1-\la)|{u^*}'(s)|^{\beta+1}).
 }
Take $\epsilon\in (0,r_0)$.  For $s\in [\epsilon, r_0]$, we have
  \aa{
\inf_{s\in [\epsilon,r_0]} \big(s^\al \min\{|u'(s)|^{\beta +1},|{u^*}'(s)|^{\beta+1}\}\big)  \lt \La(s,\la)\lt  r_0^\al|u'(r_0)|^{\beta+1}.
 }
 Therefore, there exists a constant $K>0$ such that for  $r\in [\epsilon, r_0]$,
 \aa{
\sup_{[r,r_0]} |u - u^*| \lt  K \int_{r}^{r_0}  \sup_{[s,r_0]}|u(s) - u^*(s)| ds.
}
By Gronwall's inequality, $\sup_{[\epsilon,r_0]}|u - u^*| = 0$. Since $\epsilon\in (0,r_0)$
is arbitrary, we obtain that $u=u^*$ on $(0,r_0)$ which is a contradiction. Thus, \rf{prop1} is true.

 Recall that $\la(\rho)^\frac1{\te}$ and  ${\la^*}^\frac1{\te}$ are zeros
   of $u(\fdot,\rho)$ and $u^*$, respectively. Since $u$ and $u^*$ are decreasing in $r$,
 they do not have other zeros. 
Define $r(\rho) = \min\{\la(\rho)^\frac1{\te}, {\la^*}^\frac1{\te}\}$ and
$z(\rho) = \ms L_{(0,r(\rho))}\hat v(\fdot,\rho)$. Note that if  $\rho$ is sufficiently large, then,
by   Proposition \ref{lrhoB}  , we can find two positive constants $\tau_1$ and $\tau_2$ such that
$\tau_1<r(\rho)<\tau_2$. Furthermore, we note that $z(\rho)$ is positive and finite.
Indeed, $z(\rho)>0$  since, by Theorem \ref{thm888},
$\lim_{\rho\to+\infty} \ms L_{(0,\tau_1)} (\hat v(\fdot,\rho)) = +\infty$. 
Assuming that $z(\rho)$ is infinite, we can find a sequence $\{r_n\}$ of zeros of $\hat v$ converging to
 a point $\td r\in [0,r(\rho)]$. One has $\hat v(\td r,\rho) = \pl_r \hat v(\td r,\rho) = 0$ which 
 contradicts to \rf{prop1}. 
 Thus, we conclude that  for each fixed $\rho$, the equation $\hat v(r,\rho) = 0$ has a finite number of solutions on $(0,r(\rho))$.
 By the same argument,  for each fixed $\rho>0$, 
 the equation $\hat v(r,\rho)=0$ has a finite number of solutions on any finite interval $[0,R]$, and therefore,
  there is a countable number of zeros of  $\hat v(r,\rho)$ on $(0,+\infty)$.

Fix $\rho_0>0$ and
take an arbitrary zero $r_0$ of $\hat v(\fdot,\rho_0)$.  
 By \rf{prop1}, $\pl_r \hat v(r_0,\rho_0) \ne 0$. By the implicit function theorem, 
there exists an interval $(\rho_1,\rho_2)$, containing $\rho_0$, 
and a unique continuous function $\mf r(\rho)$ solving the equation  $\hat v(r,\rho) = 0$ 
on $(\rho_1,\rho_2)$ and such that $\mf r(\rho_0) = r_0$. 
Without loss of generality, $(\rho_1,\rho_2)$ can be regarded
as the maximal interval to which  $\mf r(\rho)$ can be extended. This extension
will be unique by \rf{prop1} and the implicit function theorem.
The curve $\mf r(\rho)$ may cross $r(\rho)$ at some point from $(\rho_1,\rho_2)$ 
or may not. When it happens a zero enters or exits the interval $(0,r(\rho))$. 
Remark that $\rho_1$ can be either zero or positive;
$\rho_2$ can be either finite or $+\infty$.
Let us show that if $\rho_2$ is finite, then $\lim_{\rho\to\rho_2-} \mf r(\rho) = +\infty$.
Suppose $\rho_2$ is finite, but $\lim_{\rho\to\rho_2-} \mf r(\rho) \ne +\infty$.
 Then, we can find a sequence $\rho^n\to\rho_2$ as $n\to+\infty$
such that $\mf r(\rho^n)$ converges to a number $r_2\in (0,+\infty)$. By continuity, $\hat v(r_2,\rho_2) = 0$.
If $r_2$ is the limit of $\mf r(\rho)$ 
as $\rho\to\rho_2$, then, by the implicit function theorem, we can find an extension of $\mf r(\rho)$ 
beyond $\rho_2$ which is a contradiction. Therefore, the limit at $\rho_2$ does not exist.
Then, we can find a neighborhood $U = U(r_2,\rho_2)=(r_2-\dl,r_2+\dl) \x (\rho_2-\dl,\rho_2+\dl)$, 
where 1) the equation $\hat v(r, \rho) = 0$ is uniquely solvable for each fixed $\rho\in (\rho_2-\dl,\rho_2+\dl)$
and, furthermore, determines $r$ as a continuous function of $\rho$ living inside $U$; and 2)
$\mf r(\rho)$ exits and enters $U$ infinitely many times. This contradicts to
the implicit function theorem. Hence, if $\rho_2$ is finite, we have that $\lim_{\rho\to\rho_2-} \mf r(\rho) = +\infty$.
By the similar argument, one can show that if $\rho_1>0$, then
$\lim_{\rho\to\rho_1+} \mf r(\rho) = +\infty$.
Also remark that by the implicit function theorem and \rf{prop1}, two functions $\mf r_1(\rho)$ and $\mf r_2(\rho)$ representing
two different solutions of  $\hat v(r,\rho) = 0$  cannot exit or enter the interval $(0,r(\rho)]$ through its right end at the same point $\rho$.
One can extend this observation stating that  two functions $\mf r_1(\rho)$ and $\mf r_2(\rho)$ representing
two different solutions of  $\hat v(r,\rho) = 0$  never intersect each other, which means that two zeros cannot merge.
The above arguments allow us to conclude that $z(\rho)$ is positive,
 finite, and changes only by $1$ when a zero of $\hat v(\fdot,\rho)$ enters or exists the 
interval $(0,r(\rho))$  through its right end.

Let us prove now that $\la(\rho)$ oscillates infinitely around $\la^*$ as $\rho\to+\infty$. 
First, we note that $\hat v(0,\rho) = - \infty$ for all $\rho>0$, i.e., does not change the sign as $\rho$ varies.
Then, $z(\rho)$ changes only if $\hat v(r(\rho),\rho)$ changes the sign.
Suppose there exists $\sg_1>0$ such that
for all $\rho>\sg_1$, $\la(\rho)\gt\la^*$. Then, $r(\rho) = {\la^*}^\frac1{\te}$ and we have that 
$\hat v(r(\rho)) = u({\la^*}^\frac1{\te},\rho)\gt 0$ for all $\rho>\sg_1$. Likewise, if we assume that there exists 
$\sg_2 >0$ such that for all $\rho>\sg_2$, $\la(\rho)\lt\la^*$, then $r(\rho) = \la(\rho)^\frac1{\te}$ and
$\hat v(r(\rho)) = -u^*({\la(\rho)}^\frac1{\te},\rho)\lt 0$ for all $\rho>\sg_2$. In both situations,
i.e.,  $\la(\rho)\gt\la^*$ for all $\rho>\sg_1$ or $\la(\rho)\lt\la^*$ for all $\rho>\sg_2$, 
a zero of $\hat v$ cannot enter 
the open interval $(0,r(\rho))$ through its right end (however it can reach $r(\rho)$ and then turn back)  when $\rho>\max\{\sg_1,\sg_2\}$,
which contradicts to Theorem \ref{thm888} stating that $\lim_{\rho\to+\infty} \ms L_{(0,\tau_1)} (\hat v(\fdot,\rho)) = +\infty$. 
   \end{proof}
  
 \section{Liouville-Bratu-Gelfand problem}
 \lb{s5555}
 In this section, we consider the equation
 \aaa{
  \lb{eq-u2}
 L(u) + e^u = 0.
 }
 As before, when we deal with regular solutions, we complete equation \rf{eq-u2} with the initial conditions
 $u(0) = \rho$ and $u'(0) = 0$; the regular solution will be also denoted by $u(r,\rho)$.
 Furthermore, we let $u^*$ be the singular solution to \rf{eq-u2} given by \rf{form-sing}.

In this section, we aim to complement the results of \cite{jacob-schmitt},
 where the authors studied bifurcation diagrams for problem \rf{L},
with a study of
 the number of intersection points between the regular and singular solutions to equation \rf{eq-u2}
and the convergence of regular solutions to the singular.  Recall that equations \rf{L} and \rf{eq-u2} are equivalent
by the change of variable \rf{ch-var}.

 
 \subsection{Number of intersection points between the regular and singular solutions}
 
We already studied the number of intersection points between the regular and singular solutions $u(\rho,\fdot)$ and $u^*$
to \rf{eq-u2} in Subsection \ref{sub4.2}.  Namely, we obtained Theorem \ref{zeros1} whose statement we repeat here 
for completeness of the presentation.
\begin{customthm}{4.1}
\lb{thm5151}
Assume (A6).
Suppose $\al-\beta-1< 4\hat\te$ $(2k<d<2k+8)$. Then,  $\ms L_{(0,\infty)}(u-u^*) = +\infty$.
\end{customthm}
\begin{thm}
 \lb{zeros2}
 Assume (A6).
 Suppose $\al-\beta-1\gt 4\te(\beta+1)$ ($d\gt 2k+8k^2$). Then, $\ms L_{(0,+\infty)}(u-u^*) = 0$.
   \end{thm}
 \begin{proof}
 First of all, we note that $\al-\beta-1\gt 4\te(\beta+1)\gt 4\hat\te$. We keep the same notation as in the proof
 of Theorem \ref{zeros1} in Subsection \ref{sub4.2}, except for the notation for the regular and singular solutions. Namely, we define the same objects
 with $w=u$ and $w^*=u^*$.
 The unique steady state $(0,0)$ of \rf{linearized}, and therefore of \rf{sys1}, is then an unstable node. We further note that
 the eigenvectors for the linearized system  \rf{linearized} point to the first quadrant of the phase plane. 
 Therefore, in a small neighborhood of the origin, the integral curves are located at the first and the third quadrants. 
As $t$ increases, the motion along the curves occurs clockwise.
 Note that the solution to \rf{sys1} satisfies the following conditions: 
 \aaa{
 \lb{lim8}
 \lim_{t\to+\infty} x(t) = -\infty \quad  \text{and} \quad \lim_{t\to+\infty} y(t) = -\te^{\beta+1}.
 }
From \rf{sys1}, we will evaluate the derivative $\frac{dy}{dx}$ at the points of the curve
 $y =\mf c(x) = (\ell x+\te)^{\beta+1} - \te^{\beta+1}$, which is the same as $(y+\te^{\beta+1})^\frac1{\beta+1} - \te = \ell x$,
 where the number $\ell$ is to be fixed later. Introducing the notation $\dl = \al-\beta-1$, we obtain that
 in the area $\{x< 0$, $-\te^{\beta+1} < y < 0\}$, the following estimate holds:
 \aaa{
 \lb{in45}
 \frac{dy}{dx}\Big|_{\mf c} = \dl \Big(\frac{y}{(y+\te^{\beta+1})^\frac1{\beta+1}-\te} -\frac{\te^{\beta+1}}{\ell} \frac{e^x - 1}{x}\Big)
 > \dl\te^\beta\Big(1-\frac{\te}{\ell}\Big) \gt (\beta+1)\te^\beta \ell,
 }
 where the number $\ell$ is to be chosen in such a way that the last inequality is satisfied.
To see why the second inequality is true, we note that the function $y\mto \frac{y}{(y+\te^{\beta+1})^\frac1{\beta+1}-\te}$ is non-decreasing on 
  $(-\te^{\beta+1},0)$ (a straightforward  computation shows that it has a positive derivative if $\beta\gt 0$), and therefore,
  on $(-\te^{\beta+1},0)$,  it is greater than $\te^\beta$. Now solving the inequality
  \aa{
  \dl \Big(1-\frac{\te}{\ell}\Big)\gt (\beta+1) \ell
  }
  with respect to $\ell$, we obtain that under the assumption $\dl \gt  4\te(\beta+1)$, it has at least one solution. 
Next, we remark that  $\mf c(x) = (\ell x+\te)^{\beta+1} - \te^{\beta+1}$ is a convex increasing function on $[-\frac{\te}{\ell},0]$
whose graph connects the points $(-\frac{\te}{\ell}, - \te^{\beta+1})$ and $(0,0)$. 
  We further note that the integral curves of \rf{sys1} cannot intersect the curve $\mf c$
 when approaching it from the right. Indeed, $\frac{d \mf c}{dx} = (\beta+1)  (\ell x+\te)^\beta \ell<(\beta+1)\te^\beta \ell$ when $x<0$.
 Assuming that such an intersection occurs, by \rf{in45}, we obtain that
$\frac{dy}{dx}\big|_{\mf c} > \frac{d \mf c}{dx}$.
Therefore, the only possibility for \rf{lim8} to hold is when the integral curve, solving \rf{sys1}, never leaves the third quadrant.
This means that this curve never intersects  the $y$-axis, and therefore $v(t) \ne v^*(t)$ for all $t$. This implies that $u(\rho,r)\ne u^*(r)$
for all $r>0$.
 \end{proof}
\subsection{Convergence of regular solutions}
Here, we prove Theorem \ref{conv1122}.
\begin{proof}[Proof of Theorem \ref{conv1122}]
Fix an arbitrary interval $[r_1,r_2]\sub (0,+\infty)$
and  find a sequence $u(\fdot,\rho_n)$, $\lim_{n\to \infty} \rho_n = +\infty$, 
converging uniformly on $[r_1,r_2]$ to a singular solution $u^{**}$ to \rf{eq-u2}. Let us prove that $u^{**} = u^{*}$,
where $u^*$ is the singular solution to \rf{eq-u2} given by \rf{form-sing}.

For any $B\gt 0$, we define $R^*(B) = (u^*)^{-1}(B)$. 
Recall that $R^*(0) = (\la^*)^\frac1\te$.
From \cite{jacob-schmitt} (Theorem 3.1), we know that $R(0,\rho_n) \to R^*(0)$.

Note that $u_B(\fdot,\rho_n) = u(\fdot,\rho_n) - B$ solves the equation $L(u) + e^Be^{u} =   0$. 
This implies that $\td u_B(r,\rho_n) = u_B(e^{-\frac{B}{\te}} r,\rho_n)$ satisfies equation \rf{eq-u2} and  
the initial conditions $\td u_B(0,\rho_n) = \rho_n - B$,
$\pl_r \td u_B(0,\rho_n) = 0$.
Therefore, $u_B(e^{-\frac{B}{\te}} r,\rho_n) = u(r,\rho_n-B)$. Hence,
$u_B(e^{-\frac{B}{\te}} R(0,\rho_n-B),\rho_n) = 0$ which is the same as 
$u(e^{-\frac{B}{\te}} R(0,\rho_n-B),\rho_n) = B$. 
Thus, as $n\to \infty$, 
\aa{
R(B,\rho_n) = e^{-\frac{B}{\te}} R(0,\rho_n-B)\to e^{-\frac{B}{\te}} R^*(0) = R^*(B),
}
where the convergence holds by Theorem 3.1 in \cite{jacob-schmitt}.
To see why  $R^*(B) = e^{-\frac{B}{\te}}R^*(0)$, we observe that
$u^*_B(e^{-\frac{B}{\te}} r) =\la^* - \te \ln r =  u^*(r)$ (a verification is straightforward).
This implies that
  $u^*_B(e^{-\frac{B}{\te}} R^*(0))= 0$, which is the same as  $u^*(e^{-\frac{B}{\te}} R^*(0))= B$. 
  Hence,
 \aa{
 u^{-1}(\fdot,\rho_n)\to (u^*)^{-1} \qquad \text{in} \;\; \C_{loc}(0,+\infty). 
 }
 For simplicity, we write $u_n$ for $u(\fdot,\rho_n)$.
 Since $u_n$ and $u^{**}$ are decreasing in $r$ and
 $u_n \rightrightarrows u^{**}$ on $[r_1,r_2]$, for a small $\eps>0$, there exists $N\in \Nnu$ such that
 $[u_n(r_2),u_n(r_1)]\sub [u^{**}(r_2) - \eps, u^{**}(r_1) + \eps]$ for $n\gt N$. Since $u_n^{-1} \rightrightarrows (u^*)^{-1}$
 on  $[u^{**}(r_2) - \eps, u^{**}(r_1) + \eps]$, we obtain that uniformly  on $[r_1,r_2]$,
 \aa{
 |r-(u^*)^{-1}(u_n(r))| = |u_n^{-1}(u_n(r)) - (u^*)^{-1}(u_n(r))| \to 0 \quad \text{as} \;\; n\to\infty.
 }
 By continuity of $u^*$, 
 \aa{
 |u^*(r) - u_n(r)| \to 0  \quad \text{as} \;\; n\to\infty
 }
 uniformly  on $[r_1,r_2]$. Therefore, $u^*=u^{**}$.  
 
 Let us obtain the first limit in \rf{lim-1}.
 Suppose it is not true. Then, one can find an interval $[r_1,r_2]\sub (0,+\infty)$, 
 a sequence $\rho_n\to +\infty$, and a number $\eps>0$ such that
 $|u(r,\rho_n) - u^*(r)|>\eps$ on $[r_1,r_2]$ for all $n$. On the other hand, we can find a subsequence $u(\fdot,\rho_{n_k})$ 
 which converges to a solution of \rf{eq-u2} uniformly on $[r_1,r_2]$. 
 By Corollary \ref{cor-comp}, this solution is singular. By what was proved, this singular solution is $u^*$.
 This is a contradiction. Thus, the first limit in \rf{lim-1} is obtained.
 
 To show the second equality in \rf{lim-1}, we note that by Theorem 3.1 from \cite{jacob-schmitt}, 
 $\la(\rho) \to \la^*$ as $\rho\to+\infty$. Since $u_{{\sss \la(\rho)}}(r) = u(\la(\rho)^\frac1\te r,\rho)$ and $u^*_{\lm^{\!*}}(r) = u^*({\la^*}^\frac1\te r)$,
 the result follows.
 \end{proof}
 \begin{cor}
 Let assumptions of Theorem \ref{conv1122} be fulfilled and let $\la^\#$ be as in Theorem \ref{thm3131}.
 Then, if  $0<\al-\beta-1< 4\hat\te$ $(2k<d<2k+8)$, 
 $(u_{\lm^{\#}}, \la^{\#})$ is a regular solution;
 if $\al-\beta-1\gt  4\hat\te$   $(d \gt 2k+8)$, then $\la^\# = \la^*$.
  \end{cor}
  \begin{proof}
  If  $0<\al-\beta-1< 4\hat\te$, the statement follows from Theorems  \ref{pro189}
  and \ref{conv1122}. If $\al-\beta-1\gt  4\hat\te$, the statement follows from Theorem 3.1 (case III) in 
  \cite{jacob-schmitt}.
  \end{proof}

\subsection*{Open problems}
In this work, some of the results were obtained for the case $f(u) = u$. It would be interesting to obtain those results,
along with the convergence $\la(\rho)\to\la^*$, for problem \rf{G}.
 More specifically, the results of interest are
\begin{itemize}
\item[(i)]  For problem \rf{G}, to show that $\la(\rho)\to\la^*$ as $\rho\to+\infty$.
\item[(ii)] For the same problem, to obtain the result of Theorem \ref{conv1122} about the convergence of regular 
solutions to the singular.
\item[(iii)] For equation \rf{sing1}, to find a condition on $\al$, $\beta$, $\gm$ so that the singular and regular solutions
do not intersect each other. Such a result would be a generalization of Theorem \ref{zeros2-}.
\end{itemize}
These problems are the subject of our future research.

\subsection*{Appendix}
Remark that the radial form of problem \rf{G0} is the following:
 \eq{
 \lb{11111}
 \frac{C^k_d}{d}\, r^{1-d} (r^{d-k} w'(r)^k)' = \la e^{f(-w)},  \quad 0\lt r<1, \\
 w<0, \quad 0\lt r<1,\\
 w(1) = 0.
 }
 We show that for regular  radial solutions, \rf{11111} can be reduced to \rf{G}, and vice versa.
 Since the solutions of interests are radial in the ball, we have that $w'(0) = 0$.
 Therefore, the integral form of the above equation is 
 \aa{
 r^{d-k} w'(r)^k = \frac{\la d}{C^k_d}\int_0^r s^{d-1} e^{f(-w(s))} ds.
 }
This shows that $w'(r)$ cannot take zero values inside the ball except for the origin. Since $w(1) = 0$ and $w'(0) = 0$, 
$w(r)$ is strictly increasing, and therefore, $w'(r)>0$ for $0<r\lt 1$. 
By doing the substitution $u = -w$, we rewrite problem \rf{11111} as follows:
 \eqq{
-\frac{C^d_k}{d}\,  r^{1-d} (r^{d-k} |u'(r)|^{k-1} u'(r))' = \la e^{f(u)}, \quad 0\lt r<1,\\
 u>0, \quad 0\lt r<1,\\
 u(1) = 0.
 }
By this, in the context of regular solutions, problem \rf{G0} is equivalent to problem \rf{G} with the operator $L$ defined by \rf{L-op}
and $\al = d-k$, $\beta = k-1$, $\gm = d-1$.

On the other hand, as we showed in Subsection \ref{s223}, 
for any singular solution to \rf{G}, it holds that $(u^*_{\lm^{\!*}})' < 0$.
This implies that $w^*_\lm = - u^*_{\lm^{\!*}}$ is a singular solution to \rf{11111}.

\subsection*{Acknowledgements} 
The authors thank the anonymous referee for valuable suggestions on improvements of the readability of the paper. J.M. do \'O acknowledges partial support  from 
CNPq through the grants 312340/2021-4 and 
429285/2016-7 and 
Para\'iba State Research Foundation (FAPESQ), grant no 3034/2021.
 E. Shamarova acknowledges partial support from 
Universidade Federal da Para\'iba
 (PROPESQ/PRPG/UFPB) through the grant 
PIA13631-2020 (public calls no 03/2020 and 06/2021).

\end{document}